\pdfoutput=1
\documentclass[]{amsart}
\RequirePackage[utf8]{inputenc}
\usepackage{amsfonts,amssymb,amsmath,amscd,amsthm,amsbsy, tikz,hyperref, appendix}
\usepackage{tikz-cd}
\usepackage{pgf,tikz}
\usepackage{pgfplots}
\pgfplotsset{compat=1.15}
\usepackage{mathrsfs}
\usetikzlibrary{arrows}
\usepackage[normalem]{ulem}
\usepackage{stmaryrd}
\usepackage{setspace}
\usepackage{fix-cm}
\usepackage[paper=a4paper, dvips, top=2.5cm, left=1.8cm, right=1.8cm, foot=1.3cm, bottom=3.2cm]{geometry}
\usepackage{marginnote}
\usepackage{letltxmacro, array, cleveref}
\usepackage{mathtools}
\usepackage{multicol}
\usepackage{caption}
\usepackage{comment}
\usepackage{multirow}
\usepackage{graphicx}

\usepackage{nicematrix}
\usepackage{tabularray}
\UseTblrLibrary{booktabs, xparse} 
\newcommand{\mytableextraspace}{\addlinespace[.4em]}
\usepackage{multicol}
\usepackage{booktabs}
\usepackage{extarrows}
\usepackage{multicol}
 
\LetLtxMacro\mn\marginnote
\usepackage{enumerate}
\tikzset{node distance=2cm, auto}
\title[On the geometry and topology of positively curved Eschenburg orbifolds]{
On the geometry and topology of positively curved Eschenburg orbifolds }
\author[D.Wulle]{Dennis Wulle}
\address[D. Wulle]{Fakultät Mathematik\\
Universität	Münster\\
Münster\\
Germany.}
\email{dennis.wulle@uni-muenster.de}
\author[M. ZAREI]{Masoumeh Zarei}
\address[M. ZAREI]{Fakultät Mathematik\\
Universität	Münster\\
Münster\\
Germany.}
\email{masoumeh.zarei@uni-muenster.de}

\date{\today}

\newcommand{\Gl}{\mathrm{Gl}}

\newcommand{\Z}{\mathbb{Z}}

\newcommand{\R}{\mathbb{R}}
\newcommand{\C}{\mathbb{C}}

\newcommand{\Q}{\mathbb{Q}}

\newcommand{\EO}{\mathcal{O}_{p, q}^{a, b}} 
\newcommand{\ESO}{\mathrm{E}_{p, q}^7} 

\newcommand{\UU}{\mathrm{U}} 
\newcommand{\Sph}{\mathrm{S}} 
\newcommand{\CB}{\mathrm{B}} 

\newcommand{\T}{\mathsf{T}}

\renewcommand{\H}{\mathsf{H}}
\renewcommand{\S}{\mathsf{S}}
\newcommand{\U}{\mathsf{U}}
\newcommand{\SU}{\mathsf{SU}}

\renewcommand{\O}{\mathsf O}

\newcommand{\diag}{\operatorname{diag}}

\newcommand{\Id}{{\rm Id}}

\newcommand{\sgn}{{\rm sgn}}
\newcommand{\bb}{\mathbb}
\newcommand{\Smith}{\mathrm{Smith}}
\renewcommand{\T}{\mathrm{T}_{p,q}^{a,b}}
\newcommand{\OO}{\mathcal{O}}
\newcommand{\Fr}{\mathrm{Fr}}
\newcommand{\Conv}{\mathrm{Conv}\,}
\newcommand{\pr}{\mathrm{pr}}

\newcommand{\RN}[1]{%
	\textup{\uppercase\expandafter{\romannumeral#1}}%
}

\newtheorem{maintheorem}{Theorem}

\newtheorem{theorem}{Theorem}
\newtheorem{conjecture}{Conjecture}
\newtheorem{lemma}[theorem]{Lemma}

\newtheorem{defn}[theorem]{Definition}
\newtheorem{corollary}[theorem]{Corollary}
\newtheorem{proposition}[theorem]{Proposition}

\newtheorem{example}[theorem]{Example}

\numberwithin{equation}{section}
\numberwithin{theorem}{section}
\theoremstyle{definition}
\newtheorem{remark}[theorem]{Remark}
\begin{document}

\begin{abstract}
The present article explores the relationship between positive sectional curvature and the geometric and topological properties of Eschenburg $6$-orbifolds. First, we prove that positive sectional curvature imposes restrictions on the their singular sets, thereby confirming a conjecture posed by Florit and Ziller. Then we compute the orbifold cohomology rings for those with a specific singular locus. This reveals a distinctive behavior in the cohomology groups of positively curved Eschenburg orbifolds compared to their non-negatively curved counterparts. Furthermore, we compute the orbifold cohomology rings of all Eschenburg orbifolds.

\vspace*{-.5cm}
\end{abstract}
\vspace*{-.5cm}
\maketitle
\reversemarginpar
%

\section{Introduction}
Curvature conditions on Riemannian manifolds, and more recently on singular spaces, have  been a central interest in Riemannian geometry. This endeavor generally follows two main directions: finding examples and identifying obstructions. The obstructions are often related to specific topological properties, raising questions about the interplay between geometry and topology. Of particular  importance are the topological implications of assuming a lower curvature bound, such as positive and non-negative sectional curvature. 

{\color{black} 
While the class of non-negatively curved compact manifolds  contains many examples, including their products and quotients of compact Lie groups,  
very few examples of positively curved ones are known.  
In fact, a full list of  known simply connected examples is given by homogeneous spaces, including the compact rank one symmetric spaces, the Wallach flag manifolds and positively curved Aloff--Wallach spaces, see e.g. \cite{WilkingZiller}, two infinite families in dimensions $7$ and $13$, known as Eschenburg and Bazaikin spaces \cite{Esch82,Bazaikin1996}, respectively, and an additional example in dimension $7$ discovered more than a decade ago \cite{Dearricott2011, Grove2011}  (see also \cite{Zil14} for a survey).

This stands in a sharp contrast with    the few known obstructions which distinguish these two classes,
 primarily consisting of classic results on the fundamental groups such as Bonnet--Myers and Synge's theorems. Important conjectures  like the  Deformation Conjecture for simply connected manifolds and  the Hopf Conjecture are  wildly open due to this phenomenon. The Deformation conjecture claims that a non-negatively curved metric can be deformed to a positively curved one  if one point has positive sectional curvature. Counterexamples to this  for non-simply connected manifolds are given in \cite{Wilking2002}. The Hopf conjecture claims that compact rank two symmetric spaces do not admit metrics with positive sectional curvature. Note that under symmetry assumptions, however, additional obstructions have been found, e.g. \cite{Wilking2003, Wilking2006} and more recently \cite{Kennard2021-ge, Nienhaus22}. 
 }

In the search  for new examples of Riemannian manifolds admitting positive sectional curvature, exploring quotients of Lie groups has constantly proven to be a promising approach. In fact, among the few known examples, only one does not arise from this construction. All others are either homogeneous spaces or, more generally, \emph{biquotients} $G\sslash U$, where $G$ is a compact Lie group and $U\subseteq G\times G$  a closed subgroup which  acts on $G$ by $(u_1, u_2)\cdot g=u_1gu_2^{-1}$. {\color{black}If the action is free then $G\sslash U$ is a manifold, and if the action is almost free, i.e. all isotropy groups are finite,  it is an orbifold. Recall that  orbifolds are singular spaces which are locally modeled on quotients of finite group actions on a Euclidean space.}

The concept of biquotient actions was first introduced by Gromoll and Meyer in relation to exotic spheres \cite{GM74} and applied by Eschenburg and Bazaikin to construct manifolds with non-negative and positive sectional curvature \cite{Esch82,Bazaikin1996}. 
The Eschenburg and Bazaikin spaces, particularly their geometry and topology, have been the focus of intensive study over the past few decades (see \cite{ASTEY199741, Chinburg2007, DeVJoh23,  Esch92,  Florit2009, Kerin2011,Kreck1991}). Moreover, they contain intriguing subfamilies, such as  positively curved cohomogeneity one manifolds, which have reappeared in certain classification results \cite{gwz, Wulle2023}, or the well-known Aloff--Wallach spaces.

  The primary focus of this article is on a family of $6$-dimensional biquotient orbifolds known as \emph{Eschenburg orbifolds}, introduced by Florit and Ziller in \cite{FZ07}. They  showed that the family of $7$-dimensional  Eschenburg spaces are the total spaces of some orbifold fibrations over these spaces. {\color{black} In fact, apart from some of the compact rank one symmetric spaces,   all known compact positively curved manifolds are the total spaces of either a Riemannian submersion or an orbifold fibration.} 
This context raises the hope that studying positive curvature in a broader category may lead to finding obstructions or even discovering new examples in the smooth category.

 The Eschenburg $6$-orbifolds are defined as follows.
For $a,b,p,q\in \Z^3$ with $\sum p_i = \sum q_i$, $\sum a_i = \sum b_i$  define  
\begin{align*}
    \EO \coloneqq \SU(3)\sslash \T
\end{align*}
by an almost free biquotient action of a $2$-torus $T^2 = \lbrace (z,w) \in \bb C^2\, \vert \, z,w \in S^1\rbrace$ as  \begin{align*}
    (z,w) \ast A \coloneqq z^pw^a  A\,  \bar w ^b \bar z^q.
\end{align*}
Equivalently, $\EO$ can be defined by an isometric biquotient circle action on the $7$-dimensional Eschenburg orbifolds $E^7_{p,q}=\SU(3)\sslash \mathrm{S}^1_{p,q}$ (see Section~\ref{sec:eschorb} for details).  We equip  $\EO$ with the submersion metric of a Cheeger deformation of the bi-invariant metric on $\SU(3)$ along some $\U(2)$ subgroup. We call such metrics Eschenburg metrics (see \cite{Esch82}). With these metrics all of the orbifolds   $\EO$ are non-negatively curved and we refer to Section~\ref{sec:poscurv} for the positive curvature condition.

By \cite[Theorem A]{FZ07}, 
the singular set is contained in the union of  $9$ totally geodesic orbifold $2$-spheres $\bb S^2_{ij}$ arranged according to the following diagram \cite[Theorem~A]{FZ07}: \vspace*{-.5cm}
\begin{figure}[ht!] \caption{The structure of the singular set}\label{fig:sing:locus}\centering
\begin{picture}(100,125)(160,30)  
\small \put(220,140){\line(-2,-1){50}}
\put(220,140){\line(2,-1){50}} \put(220,140){\line(0,-1){100}}
\put(220,40){\line(-2,1){50}} \put(220,40){\line(2,1){50}}
\put(270,115){\line(0,-1){50}} \put(170,115){\line(0,-1){50}}
\put(170,115){\line(2,-1){100}} \put(270,115){\line(-2,-1){100}}
\put(215,148){$\ast_{\Id}$} \put(148,54){$\ast_{(123)}$}
\put(148,122){$\ast_{(23)}$} \put(270,122){$\ast_{(12)}$}
\put(270,54){$\ast_{(132)}$} \put(210,30){$\ast_{(13)}$}
\put(240,40){$\bb S^2_{13}$} \put(186,40){$\bb S^2_{31}$}
\put(273,88){$\bb S^2_{21}$} \put(150,88){$\bb S^2_{23}$}
\put(184,135){$\bb S^2_{11}$} \put(240,135){$\bb S^2_{33}$} \put(202,60){$\bb S^2_{22}$} \put(188,107){$\bb S^2_{32}$} \put(244,95){$\bb S^2_{12}$}
\put(217.3,137){$\bullet$} \put(217.3,37){$\bullet$}
\put(167.4,112){$\bullet$} \put(267.3,112){$\bullet$}
\put(267.3,62){$\bullet$} \put(167.4,62){$\bullet$} 
\end{picture}
\end{figure}
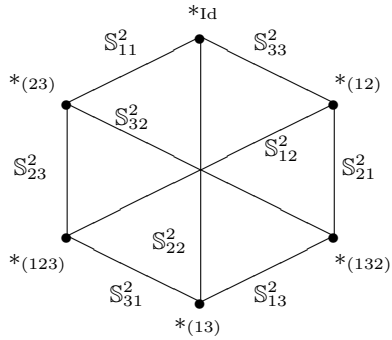

\noindent
Here $(ij)$ denotes the permutation interchanging $i$ and $j$ and $(ijk)$ the cyclic permutation $i \mapsto j \mapsto k\mapsto i$. The vertices $\ast_\sigma$ correspond to   points in $\EO$ and  the edges $\bb S^2_{ij}$ correspond to orbifold  $2$-spheres   connecting precisely two vertices. 
We say that two vertices $\ast_\sigma$ and $\ast_\tau$ have the same parity, if the permutations $\sigma$ and $\tau$ have the same parity. In general the singular set $\Sigma$ of $\EO$ is represented by  a subgraph of \eqref{fig:sing:locus}. It can contain isolated points $\ast_\sigma$ or the  entire edge $\bb S^2_{ij}$. Note that the inner points of an edge always have the same local groups. For the vertices,  there are three possibilities: First, they both have  the same local groups as the inner points, i.e.  the edge is  a ``smooth sphere''. Second, exactly one vertex has a different local group than the interior, i.e. the edge is a ``teardrop''. Third, both vertices have different local groups than the interior, i.e.  the edge is a ``football''. We give more details on the singular sets in Subsection~\ref{SS:singular set}.

The first result addresses the structure of the singular set of Eschenburg orbifolds with positive curvature, a topic initially explored in \cite{FZ07} and further examined in \cite{Y14} and \cite{Y15}. Recall that there are only two smooth manifolds within the family of Eschenburg 6-orbifolds:  the Wallach flag manifold $W^6_1 = \SU(3)/T^2_{\max}$ and the twisted Wallach flag $W^6_2 = \SU(3)\sslash T^2$.  
This implies that, generically, we should expect a non-empty singular set for $\EO$. In fact, various subgraphs of Figure~\ref{fig:sing:locus} can occur (see \cite[p. 174]{FZ07} and \cite[Theorem~D]{Y14}). Given that $W^6_1$, $W^6_2$ admit metrics with positive curvature,  it raises the natural question of whether curvature conditions impose restrictions on the singular set. This was also addressed in \cite{FZ07} and based on their result on positively curved cohomogeneity one Eschenburg $7$-manifolds, the authors conjectured:
\begin{conjecture}\cite{FZ07}\label{Conj:FZ}
    None of the positively curved Eschenburg $7$-manifolds admit an isometric circle action for which $\EO$ has only one singular point.
\end{conjecture}
Note that Conjecture \ref{Conj:FZ} does not hold for non-negative curvature, or even almost positive curvature \cite[p. 172]{FZ07}. 
Theorem~\ref{thm:main:sing} provides more insights into the structure of the singular sets of positively curved Eschenburg orbifolds, and consequently confirms the conjecture.

\begin{maintheorem}\label{thm:main:sing}
 Let $\EO$ be an effective positively curved Eschenburg orbifold with non-empty singular set $\Sigma$. Then the following statements hold.
\begin{enumerate}
    \item  \label{thm:main:sing:same_Parity}  $\Sigma$ contains at least one  vertex with odd  and one  vertex with even parity.

    \item \label{thm:main:sing:smoothsphere} 
    If $\Sigma$ contains exactly two vertices, then it also contains the sphere connecting them. Furthermore, the sphere is smooth.
\end{enumerate}
\end{maintheorem}

    Theorem~\ref{thm:main:sing} implies that $\Sigma$ is either empty or  consists of at least three points. We remark that this conclusion  is  sharp: There are positively curved examples with exactly three singular points (see \cite[Theorem~5.2.10]{Y14}). Moreover, there are Eschenburg $6$-orbifolds with almost positive curvature and with exactly one or exactly two singular points, Example~\ref{ex:nonegative}.  
    
    From the proof of Part~\ref{thm:main:sing:smoothsphere} of Theorem~\ref{thm:main:sing}, we present a full classification, up to equivalence, of the spaces where the singular set is contained in one edge of graph~\eqref{fig:sing:locus}.  Corollary~\ref{cor:nonnegact} includes all orbifolds with precisely one singular point, while Corollary~\ref{lem:actionssmooth} covers all cases where the singular set consists of exactly one smooth sphere. In the latter situation, a distinction between positive and non-negative sectional curvature is the existence of  cohomogeneity two actions. In particular, under the assumptions of Theorem~\ref{thm:main:sing}~\eqref{thm:main:sing:smoothsphere}, the orbifold $\EO$ admits an Eschenburg metric with positive sectional curvature if and only if it admits a cohomogeneity two action by $\SU(2)\cdot S^1\subset \SU(3)$. In general an Eschenburg orbifold only admits a cohomogeneity $4$ torus action.\\

In the the second part of the article, we focus on the topology of Eschenburg orbifolds. We are especially interested in the orbifold cohomology ring. In general, the definition  is delicate, but in this special case it is  given by the equivariant cohomology of the $\T$-action on $\SU(3)$:
\begin{align*}
    H^\ast_\OO(\EO) = H^\ast_{\T}(\SU(3)) \coloneqq H^\ast(E\T \times_{\T} \SU(3)),
\end{align*}
where $E\T \times_{\T} \SU(3)$ is the Borel construction, see Section \ref{subsec:orbicohom} for more details. In the light of Theorem~\ref{thm:main:sing}, it is particularly interesting to study those examples where the singular set consists of a smooth sphere.
Theorem~\ref{thm:main:cohomgroups} demonstrate  how positive curvature impacts the topology of these spaces through their cohomology groups.

\begin{maintheorem}\label{thm:main:cohomgroups}
    Let  $\EO$ be an effective Eschenburg orbifold with positive sectional curvature such that  ${\color{black}\bb S^2_{ij}}$ {\color{black}contains the singular set $\Sigma$ of $\EO$}. Then $\bb S^2_{ij}$ is a smooth sphere with orbifold group $\bb Z_k$, and the even degree orbifold cohomology groups stabilize  with order $k^2$.\\
   More  precisely, the cohomology groups are given by
           \begin{align*}
                H^i_{\mathcal{O}}(\EO) = \begin{cases}
                0 & \text{for } i \text{ odd}\\
                    \bb Z & \text{for } i = 0\\
                    \bb Z^2 &\text{for } i = 2,4,\\
                \end{cases}
            \end{align*}
            and
             \begin{align*}
              H^i_{\OO}(\EO) = \begin{cases}
                    \bb Z \oplus \Z_k & \text{for } i = 6\\
                    \bb Z_k^2 &\text{for } i=2m \ge 8, 
                \end{cases} \quad \text{ or }\quad   H^i_{\OO}(\EO) =  \begin{cases}
                    \bb Z \oplus \bb Z_{\gcd(2,k)} & \text{for } i = 6\\
                    \Z_{\gcd(2,k)} \oplus \bb Z_\frac{k^2}{\gcd(2,k)} &\text{for } i=2m \ge 8.
                \end{cases}
        \end{align*}
\end{maintheorem}
Notably,  we obtain a distinctive behaviour of the cohomology groups of  positively curved Eschenburg orbifolds, where the singular set is contained in {\color{black} an edge of graph \eqref{fig:sing:locus}}, as opposed to their non-negatively curved counterparts. Namely, the stable even cohomology group is finite with square order. While this condition is not sufficient for positive curvature, there are large classes of Eschenburg orbifolds with just non-negative sectional curvature fulfilling the other assumptions of Theorem~\ref{thm:main:cohomgroups}, for which the conclusion fails, see Example~\ref{ex:nonegative}. 
Finally, we note that the order $k$ of the local group of the singular set can be computed from the parameters.

After addressing the cohomology groups in certain specific examples, Theorem~\ref{thm:main:cohom} presents the full structure of the orbifold  cohomology ring of any Eschenburg $6$-orbifolds:

\begin{maintheorem}\label{thm:main:cohom} Let $\EO$ be an effective Eschenburg orbifold, then
 \begin{align*} 
     H_\mathcal{O}^{\ast}(\EO) = \bb Z[s,t]/\langle\sigma_i(s\cdot p+t\cdot a)-\sigma_i(s\cdot q +t \cdot b) \,\vert \, i=2,3\rangle,
 \end{align*}
where $\sigma_i$ are the elementary symmetric polynomials in three variables.  In particular, 
\begin{align*}
    \sigma_2(T_1,T_2,T_3) &=T_1T_2+T_2T_3+T_3T_1,\\
    \sigma_3(T_1,T_2,T_3) &=T_1T_2T_3.
\end{align*}
\end{maintheorem}
\noindent

The simplicity of the ring should not  imply  the misleading impression that deriving an explicit description of the cohomology groups is straightforward. 
While it may be relatively easy to compute all cohomology groups for given parameters $a, b, p, q, \in \Z^3$, 
doing so for all Eschenburg orbifolds or infinite subfamilies at once appears to be a  challenging problem. However, as Theorem~\ref{thm:main:cohomgroups} demonstrates, 
in specific cases we were able to compute the orbifold cohomology groups of  large  families.

Before we proceed to the strategy of the proofs, let us note that  studying orbifolds for their own sake remains an intriguing standalone problem.
The geometry of orbifolds can often be investigated by naturally generalizing concepts from smooth manifolds. We refer the reader to \cite[Chapters 3, 4]{Car22} for a detailed account of some generalizations of classical results on Riemannian manifolds, particularly in relation to positive curvature, and to \cite{Getal24, KL14,  Lange2018, 
 Stanhope2005, Yer14} for more interesting results on the geometry of orbifolds. 
While geometric problems on orbifolds can often be addressed more directly, their algebraic topology is more delicate. To capture their local structure, one can use the language of ``groupoids'', which creates a bridge between differential geometry and homotopy theory (see \cite{ALR07}). From this perspective, it is  interesting to study orbifold cohomology theories. In particular, it is appealing to develop a  a GKM theory that  connects the ``graph cohomology'' to the orbifold cohomology, see  \cite{GW22} for a treatment on the rational GKM theory for orbifolds. We point  out that, in the context of our article,  graph~\eqref{fig:sing:locus} corresponds to the ``GKM-graph'' of a $T^2$-action on $\EO$.

We briefly explain the strategies of the proofs.
The proof of Theorem~\ref{thm:main:sing} has two pillars: The first one is our geometric interpretation of the positive curvature condition on the spaces $\EO$ (Proposition~\ref{prop:pos_curv_condition}) which links the curvature properties of $\EO$ to the intersection of certain  triangles in the Euclidean plane arising from the parameters (see Figure~\ref{fig:triangles}). The second one is a collision detection method based on the \emph{Separating Axis Theorem}~\ref{Thm:Separation}, which provides a method to determine whether   the triangles intersect. The assumptions of Theorem~\ref{thm:main:sing} impose restrictions on the parameters making the method feasible. The novelty in this approach is the interpretation of the positive curvature condition as a problem in Euclidean geometry, allowing us to apply a large and well studied toolbox to solve it.
  We remark that in computational geometry the method of collision detection is used to computationally determine whether two objects are intersecting (``colliding'') or not.  In fact, it has various applications in physical modelling, computer animation, engineering and game programming \cite{lin1997collision,Lazaridis2021}.

To prove Theorem~\ref{thm:main:cohom}, we follow 
 Eschenburg's method \cite{Esch92} to determine the cohomology of a biquotient $G\sslash U$ using the spectral sequence corresponding to the bundle   
\begin{align*}
    G \to EU \times_U G \to BU.
\end{align*}
 Since the orbifold cohomology of $G\sslash U$ is defined to be the cohomology of $ EU \times_U G$,  we obtain the result.
 
Finally, for proving Theorem~\ref{thm:main:cohomgroups}, we will use Proposition~\ref{prop:isom} to relate the cohomology of $\EO$ to the cohomology of its singular set $\Sigma$. Therefore, the proof can be reduced to determining the orbifold cohomology of $\Sigma$, which is a smooth sphere by Theorem~\ref{thm:main:sing}. The smoothness simplifies the necessary computations, leading to the proof of Theorem~\ref{thm:main:cohomgroups}.  

The paper is structured as follows. Section \ref{sec:eschorb} provides a brief summary of orbifolds in general, followed by fundamental information about Eschenburg orbifolds. Section \ref{sec:poscurv} examines the positive curvature condition on Eschenburg 6-orbifolds from a geometric perspective and establishes the foundations for developing a strategy to prove Theorem~\ref{thm:main:sing}. The proof of this theorem is presented in Section~\ref{sec:singset}. Section~\ref{sec:cohom} focuses on the cohomology of Eschenburg orbifolds, beginning with an introduction to orbifold cohomology and Eschenburg's method for computing the cohomology of biquotients \cite{Esch82}. This section concludes with the proofs of Theorem~\ref{thm:main:cohom} and Theorem~\ref{thm:main:cohomgroups}.\\

\noindent\emph{Acknowledgements:} We wish to thank Burkhard Wilking for his suggestion on a problem related to Eschenburg $6$-orbifolds, which eventually led to the current project. We would like to thank Christoph B\"ohm for his  suggestions on improving the exposition of the manuscript.  We are also grateful to  Michael Wiemeler and Jan Nienhaus  for informative discussions and comments. Finally, we thank Achim Krause for his idea on the proof of Proposition \ref{prop:isom}, which resulted in a more elegant proof of Theorem \ref{thm:main:cohomgroups}.

The first named author acknowledges support by the Alexander von Humboldt Foundation through Gustav
Holzegel’s Alexander von Humboldt Professorship endowed by the Federal Ministry of Education and Research. The second named author acknowledges support  from Deutsche Forschungsgemeinschaft (DFG, German Research Foundation) under Germany's Excellence Strategy EXC 2044-390685587, Mathematics M\"unster: Dynamics-Geometry-Structure, and from  DFG grant ZA976/1-1 within the Priority Program SPP2026 "Geometry at Infinity".

\section{An overview of Eschenburg $6$-orbifolds} \label{sec:eschorb}
In this section, we first provide a brief overview of some general aspects of orbifolds. Then, following \cite{FZ07, Y14, Y15}, we recall the basic facts about Eschenburg orbifolds, establish the notation, and set up the framework for the  remainder  of this article. 

\subsection{Orbifolds} \label{subsec:orbifolds} Classically, an $n$-dimensional smooth orbifold  $\mathcal{O}$ is defined as a second-countable,
Hausdorff topological space $\lvert \mathcal{O}\rvert$, known as the \emph{underlying topological space} of $\mathcal{O}$, locally described by charts $(\tilde{U}, \Gamma, U, \pi)$, where $\tilde{U}$ is an open subset of $\R^n$, $\Gamma$ is a finite group acting smoothly and \emph{effectively} on  $\tilde{U}$, $U$ is an open subset of $X$, 
and $\pi :\tilde{U}\to U$ is a homeomorphism, along with a compatibility condition on the charts (see \cite[p. 2]{AF24} for more precise definition). 
Orbifolds are singular spaces, and their singular set is defined using the so-called local groups:
\begin{defn}
Let $x\in \lvert \mathcal{O}\rvert$ and  $(\tilde{U}, \Gamma, U, \pi)$ be 
any local chart around $x=\pi(p)$.  We define the \emph{local group} at $x$ as
\[ \Gamma_x= \{g\in \Gamma  \,\colon \,gp = p\}.\]
This group is uniquely determined up to conjugacy in $\Gamma$.
\end{defn}

\begin{defn}
For an orbifold $\mathcal{O}$, its \emph{singular set} is defined as
\[\Sigma(\mathcal{O}) =\{x\in  \lvert \mathcal{O}\rvert \,\colon\, \Gamma_x \neq 1\}.\]
We call  $R( \mathcal{O}) = \lvert \mathcal{O}\rvert\setminus \Sigma(\mathcal{O})$ the \emph{regular set} of $\mathcal{O}$.
\end{defn}

Compact transformation groups provide a rich class of examples for orbifolds, which is particularly important for this article.

\begin{defn}\label{dfn:orbifoldquotient}
    A quotient orbifold $\mathcal{O}$ is an almost free action of a Lie group $G$ on a smooth manifold $M$. Its geometric realisation is given by $\vert \mathcal{O}\vert = M/G$.  Let $\pi\colon M\to M/G$ be the projection map.  The singular set is the set  of all points $x\in M/G$ where the isotropy group $G_{p}$ is non-trivial for $\pi(p)=x$. For simplicity we often write $\mathcal{O} = M/G$, if the action is clear from the context. If the action is not effective, we call $\mathcal{O}=M/G$ an ineffective quotient orbifold.
\end{defn}

\begin{remark}\label{rem:orbifoldquotient}\mbox{}
\begin{enumerate}
\item Every  classical orbifold  can be represented as the quotient of a smooth manifold by an effective almost free action of a compact Lie group. In fact, the frame bundle $\Fr(\OO)$ of $\OO$ is a smooth manifold with an effective and almost free $\O(n)$-action such that $\Fr(\OO)/\O(n) \cong \OO$ (see e.g. \cite[Theorem 1.23]{ALR07}). For ineffective orbifolds, in general, we do not derive an equivalent description as $\O(n)$-quotients of the frame bundle. 
    \item  For brevity, we will sometimes drop the term \emph{quotient}. Further, all orbifolds mentioned in this article are quotient orbifolds.
   
    \item By definition all points in an ineffective orbifold are generically singular, since their local groups always contain the ineffective kernel. Sometimes, it is convenient to define the singular set to be the set of points, whose local groups are different from the ineffective kernel of the action.
    \item Any ineffective quotient orbifold $\mathcal{O}$ can be turned into an effective orbifold by dividing the ineffective kernel of the group action. We denote the effective structure by $\mathcal{O}^{\mathrm{eff}}$.
    \item In general an (ineffective) orbifold is defined to be a proper étale Lie groupoid. We refer the interested reader to \cite{ALR07} for a detailed introduction to this topic. It is not known whether any orbifold given by this groupoid description can be realised as a quotient orbifold.
  
\end{enumerate}
   
\end{remark}

With this remark we conclude this subsection about orbifolds and will return to some general aspects of orbifolds in Section~\ref{sec:cohom}.

\subsection{Basic Properties of Eschenburg Orbifolds} \label{subsec:eschorb:prop} 
 We begin this subsection with the definition of Eschenburg $7$-orbifolds. For $z\in \mathrm{S}^1=\{z\in \C\mid \lvert z
 \rvert=1\}$ and  $p=(p_1,p_2, p_3)\in \Z^3$, let $z^p\coloneqq\diag (z^{p_1},z^{p_2}, z^{p_3})\in \U(3)$. Given $(p, q)\in \Z^3\times \Z^3$ with $\sum p_i =\sum q_i$, 
the biquotient action of $\mathrm{S}^1$ on $\SU(3)$ is  defined as $$z\ast g\coloneqq z^pg\bar{z}^q.$$ 
We denote the quotient space  by $\ESO = \SU(3)\sslash \mathrm{S}^1_{p,q}$. The action is free if and only if 
\begin{align*}
    \gcd(p_1 -q_{\sigma(1)}, p_2-q_{\sigma(2)}) = 1 \text{ for any } \sigma \in S_3.
\end{align*}
Hence in this case $\ESO$ is a $7$-dimensional manifold. If we relax this condition to be almost free, which is equivalent to
\begin{align*}
    \gcd(p_1 -q_{\sigma(1)}, p_2-q_{\sigma(2)}) \neq 0 \text{ for any } \sigma \in S_3,
\end{align*}
then $\ESO$ is an orbifold of dimension $7$. 

We remark that if the freeness of the action is not particularly important to us, we simply refer to 
$\ESO$ as an Eschenburg orbifold, regardless of whether the action is free or almost free.

Now for $(a, b)\in \Z^3\times \Z^3$, with $\sum a_i =\sum b_i$, consider the action of $\rm S^1$ on $\ESO$ given by
\begin{align*}
  z\ast \lbrack g\rbrack\coloneqq \lbrack z^a g\bar{z}^b\rbrack.  
\end{align*} 
This action is well-defined because the $\rm S^1_{p, q}$-action on $\SU(3)$ commutes with the $\rm S^1_{a, b}$-action.  Alternatively,  we can define a biquotient  $\T$-action on $\SU(3)$ by 
\begin{align*}
  (z, w)\ast  g\coloneqq  z^pw^a g\bar{z}^q\bar{w}^b.  
\end{align*}
Note that the $\T$-action on $\SU(3)$ is almost free if and only if the $\rm S^1_{a, b}$-action on $\ESO$ is almost free. By \cite[Theorem~A]{FZ07}, the $\T$-action on $\SU(3)$ is almost free if and only if \begin{align*}
    p-q_\sigma \text{ and } a-b_\sigma \text{ are linearly independent, for all }\sigma \in S_3. 
\end{align*}
In this case, the resulting quotient space $\SU(3)\sslash \T$ is an orbifold. We call this orbifold an Eschenburg $6$-orbifold and denote it by $\EO$. 

Before we proceed, we recall some operations that yield an equivalent Eschenburg orbifold. It is worth noting that these operations are quite useful for simplifying the proofs of the main theorems by converting the torus parameters into more tractable ones.

\begin{proposition}\cite[Proposition~5.3.2]{Y14}\label{Prop:equivalent}
\begin{enumerate}
    \item $\mathcal{O}_{p,q}^{a,b} \cong \mathcal O_{p^\prime,q^\prime}^{a^\prime, b^\prime}$, whenever $(p^\prime, a^\prime, q^\prime, b^\prime)$ equals the following parameters:
	\begin{enumerate}
		\item \label{Prop:equivalent1} $(q,b,p,a)$
		\item \label{Prop:equivalent3} $(p+\bar d, a+ \bar c, q+\bar d, b + \bar c )$, for $d,c\in \mathbb Z$ and $\bar d= d \cdot (1,1,1)$
		\item  \label{Prop:equivalent4} $(p_\sigma,a_\sigma,q_\tau,b_\tau)$ for $\sigma, \tau \in \mathcal S_3$
		\item \label{Prop:equivalent5} For some $A \in \Gl(2,\mathbb Z)$:
            \begin{align*}
              \begin{pmatrix}
                    p^\prime \\a^\prime
                \end{pmatrix} = A\cdot \begin{pmatrix}
                    p \\a
                \end{pmatrix}, \quad \begin{pmatrix}
                    q^\prime \\b^\prime
                \end{pmatrix} = A \cdot \begin{pmatrix}
                    q\\b
                \end{pmatrix}.
            \end{align*}
	\end{enumerate}
 \item \label{Prop:equivalent2} $(\mathcal{O}_{p,q}^{a,b})^{\mathrm{eff}} \cong \mathcal (O_{p^\prime,q^\prime}^{a^\prime, b^\prime})^{\mathrm{eff}}$, if we additionally allow     $(\mu p, \lambda a, \mu q, \lambda b)$ for $\mu, \lambda \in \mathbb Q\backslash\lbrace 0 \rbrace$

\end{enumerate}  
\end{proposition}

\begin{remark}\label{Rem:eff_orb} Let $\EO$ be an ineffective Eschenburg orbifold. By Proposition \ref{Prop:equivalent}, we can find parameters $p^\prime, q^\prime,a^\prime,b^\prime$ such that $(\EO)^{\mathrm{eff}} \cong \mathcal{O}_{p^\prime,q^\prime}^{a^\prime,b^\prime}$ (see \cite[Corollary 5.3.3]{Y14}).
   Note further that $\gcd(p_1, \ldots, q_3)=1=\gcd(a_1, \ldots, b_3)$, if $\T$ acts effectively on $\SU(3)$.
\end{remark}

Recall that when two of the integers in $p$ or in $q$  of an  Eschenburg orbifold $\ESO$  are equal, then there exists  an $\SU(2)\times\mathrm{T}^2$ action  on $\SU(3)$  commuting with the $\S^1_{p, q}$ action, resulting in a $2$-dimensional quotient. In other words, in this case, $\ESO$ admits a cohomogeneity $2$ action.  Up to equivalence, the parameters are given by $$p=(c, d, e)\quad q=(c+d+e, 0, 0).$$
The action is free if and only if $c, d, e$ are pairwise relatively prime \cite[p.~5]{Chinburg2007}. 

In the following remark we show that every Eschenburg $6$-orbifold can be obtained as a quotient space of an  isometric almost free biquotient action of an  $\S^1$ on a cohomogeneity $2$ Eschenburg $7$-orbifolds:
 
\begin{remark}\label{R:Change_of_Parameteres}
   By  Operations~\eqref{Prop:equivalent3}  and \eqref{Prop:equivalent5} one can find the following representation for  $\EO$:
\begin{align}\label{EQ:Simpler_reparametrization}
		p = (c,d ,e), \quad q=(c+d+e, 0,0) \quad a = (a_1, a_2, a_3), \quad b=(0,b_2,b_3).
	\end{align}
 More precisely, let $g = \gcd(q_2-q_3, b_2 - b_3)$, and $\alpha = (q_2 - q_3)/g$ and $\beta = (b_3-b_2)/g$. Then $\gcd(\alpha,\beta) = 1$ and there exist $r,s \in \mathbb Z$, such that $r\beta-s\alpha = 1$. We set 
	\begin{align*}
		A = \begin{pmatrix}
			\beta &\alpha  \\
			s & r
		\end{pmatrix}. 
	\end{align*}
	 By applying $A$ to the parameters, we get $q^\prime_2 = q^\prime_3$. By Part~\eqref{Prop:equivalent3} we then obtain the parameters given in \eqref{EQ:Simpler_reparametrization}.
\end{remark}

\subsection{On the singular set of $\EO$}\label{SS:singular set}
In this subsection we give a description of the singular set  $\Sigma(\EO)$ of the Eschenburg orbifold.  

By \cite[Theorem~A]{FZ07}, the singular set $\Sigma(\EO)$ is the union of at most nine
orbifold $2$-spheres and six points which are arranged as depicted in Figure~\ref{fig:sing:locus}.

Let  $\pi \colon \SU(3) \to \EO$ be the quotient map. The orbifold $2$-spheres $\bb S^2_{ij}$ from Figure \eqref{fig:sing:locus} lifts to   $\U(2)_{ij}$ in $\SU(3)$ under $\pi$, where  
\begin{align}\label{Eq:U(2)_ij}
    \U(2)_{ij} = \left\lbrace \tau_i \begin{pmatrix}
        A &0\\0& \overline{\det A}\end{pmatrix} \tau_j\, \colon \, A \in \U(2)
     \right \rbrace
\end{align}
with  $\tau_r \in \O(3)$ a linear map interchanging the $r$-th and the third vector of the canonical basis for $r = 1,2$ and $\tau_3 = -I$. Furthermore, the six points $\ast_\sigma$ lift to the tori $T_\sigma$ given by
\begin{align*}
    T_\sigma = \mathrm{sgn}(\sigma)\sigma^{-1} \diag(z,w,\bar z \bar w) \text{ with } \sigma \in S_3.
\end{align*}
 Any $\U(2)_{ij}$ contains exactly two tori $T_{\sigma}$ and $T_{\sigma^\prime}$ with $\sigma(i) = \sigma^\prime(i) = j$ and every $T_{\sigma}$ is contained in exactly three $\U(2)_{ij}$. We say that two vertices $\ast_\sigma$ and $\ast_\tau$ have the same parity, if the permutations $\sigma$ and $\tau$ have the same parity. 

  The action of  $\T$  on $\U(2)_{ij}$    is equivalent to the biquotient  action 
\begin{align*}
   (z,w)\ast A = z^{(p_{i_1},p_{i_2})}w^{(a_{i_1},a_{i_2})}A\,\bar w^{(b_{j_1},b_{j_2})}\bar z^{(q_{j_1},q_{j_2})} = \begin{pmatrix}
       z^{p_{i_1}-q_{j_1}} w^{a_{i_1}-b_{j_1}} a_{11} &  z^{p_{i_1}-q_{j_2}} w^{a_{i_1}-b_{j_2}}a_{12}\\
        z^{p_{i_2}-q_{j_1}} w^{a_{i_2}-b_{j_1}}a_{21}&  z^{p_{i_2}-q_{j_2}}w^{a_{i_2}-b_{j_2}} a_{22}
   \end{pmatrix}
\end{align*}
on $\U(2)_{33}$, where $\lbrace i,i_1,i_2\rbrace = \lbrace j,j_1,j_2\rbrace = \lbrace 1,2,3\rbrace$. This action will appear in Section~\ref{sec:cohom}, where we compute the orbifold cohomology groups of $\bb S^2_{ij}$. 

Note that the singular set of an Eschenburg orbifold $\EO$ is possibly a subset of the graph shown in Figure~\eqref{fig:sing:locus}. To differentiate the entire graph from the actual  singular set, we use the notation $C(\EO)\coloneqq \bigcup_{i,j} \pi(\U(2)_{ij})$ for the complete graph. 

To understand the singular set of $\EO$,
 we need to examine the isotropy groups of $\U(2)_{ij}$ and $T_\sigma$
  and the ineffective kernel of the action.
Let $V,W \in \bb Z^n$, then we define
\begin{align*}
   g(V,W) \coloneqq \gcd(v_1, \ldots, v_n, w_1, \ldots, w_n) \text{ and } N(V,W) \coloneqq \gcd{\lbrace v_iw_j - v_j w_i \, \vert\, 1 \le i<j\le n\rbrace} 
\end{align*}
Note that the action on $\SU(3)$ is given as follows:
\begin{align*}
    z^pw^a A \bar w^b \bar z^q = \begin{pmatrix}
        z^{p_1-q_1}w^{a_1-b_1}a_{11} &z^{p_1-q_2}w^{a_1-b_2}a_{12}&z^{p_1-q_3}w^{a_1-b_3}a_{13}\\
        z^{p_2-q_1}w^{a_1-b_1}a_{21}&z^{p_2-q_2}w^{a_2-b_2}a_{22}&z^{p_2-q_3}w^{a_2-b_3}a_{23}\\
        z^{p_3-q_1}w^{a_3-b_1}a_{31}&z^{p_3-q_2}w^{a_3-b_2}a_{32}&z^{p_3-q_3}w^{a_3-b_3}a_{33}
    \end{pmatrix}.
\end{align*}
Since $p_i - q_i = q_{i_1}-p_{i_2} + q_{i_2}-p_{i_1}$ for $\lbrace i,i_1,i_2\rbrace = \lbrace 1,2,3\rbrace$, it  is clear  that $(z,w)$ is in the kernel of the action if and only if
\begin{align*}
    z^{p_i-q_j}w^{a_i-b_j} = 1 \text{ for all }i\neq j.
\end{align*}
This means that the kernel of the action is given by the kernel of the map
\begin{align*}
    \phi\colon T^2 \to T^6, \, (z,w) \mapsto (z^{p_i-q_j}w^{a_i-b_j})_{i\neq j},
\end{align*} 
which is given by
 $$\{(\alpha, \beta)\in \R^2\mid d_e\phi ((\alpha,\beta)) \in \Z^6\}/\Z^2,$$
 
We set $P \in \bb Z^6$ to have the components $p_i-q_j$ and $A \in \bb Z^{6}$ to have $a_i-b_j$ with $i\neq j$ and the same index ordering for both. Then $d_e\phi$ is a $6 \times 2$--matrix with columns given by $P$ and $A$. There exist $B \in \Gl_2(\Z)$ and $C\in \Gl_6(\Z)$, such that $C \cdot d_e\phi \cdot B$ has the following diagonal form, which is also known as the Smith Normal Form of $ d_e\phi$. The first diagonal entry is given by the $\gcd$ of the $1$-minors, that is $g(P,A)$, and the second one, by  $N(P,A)/g(P,A)$, since $N(P,A)$ is the $\gcd$ of the $2$-minors:
\begin{align*}
    \begin{pmatrix}
        g(P,A) &0\\
        0 & N(P,A)/g(P,A)\\
        \vdots & \vdots \\
        0&0
    \end{pmatrix}
\end{align*}
Hence the ineffective kernel is isomorphic to 
\begin{align*}
    \bb Z_{g(P,A)} \oplus \bb Z_{N(P,A)/g(P,A)}.
\end{align*}
Similarly, we recover the isotropy groups at $\U(2)_{ij}$: For this, we set  $V=(p_{i_1}-q_{j_1},p_{i_1}-q_{j_2},p_{i_2}-q_{j_1},p_{i_2}-q_{j_2})$ and $ W=(a_{i_1}-b_{j_1},a_{i_1}-b_{j_2},a_{i_2}-b_{j_1},a_{i_2}-b_{j_2})$. Let $g_{ij} = g(V,W)$ and $N_{ij} = N(V,W)$,  the isotropy group on $\U(2)_{ij}$ is given by
\begin{align*}
    \Gamma_{ij} = \bb Z_{g_{ij}} \oplus \bb Z_{N_{ij}/g_{ij}}.
\end{align*}
By \cite[Theorem 5.3.7]{Y14} or the same argument as above applied to $V = (p_1-q_{\sigma(1)}, p_2-q_{\sigma(2)})$ and $W = (a_1-b_{\sigma(1)}, a_2-b_{\sigma(2)})$ the isotropy groups at the tori $T_\sigma$ are  given by 
\begin{align*}
    \Gamma_\sigma = \Z_{g_\sigma} \oplus \Z_{N_\sigma/g_\sigma}, 
\end{align*} 
 where 
$$g_\sigma = \gcd(p_1 -q_{\sigma(1)}, p_2 -q_{\sigma(2)},a_1 -b_{\sigma(1)},a_2 -b_{\sigma(2)}).$$ 
and
\begin{align}\label{Eq:order_isotropy}
    N_{\sigma}=\lvert (p_{1}-q_{\sigma(1)})(a_{2}-b_{\sigma(2)})-(p_{2}-q_{\sigma(2)})(a_{1}-b_{\sigma(1)})\rvert\, \text{ for } \sigma\in S_3. 
\end{align}
By \cite[Remark~3.11]{FZ07}, the action of $\T$ on $\SU(3)$ is almost free if and only if $ N_{\sigma}\neq 0$, for all $\sigma  \in S_3$.

\begin{remark}
     The isotropy groups are not affected by the operations of Proposition \ref{Prop:equivalent}, except for \eqref{Prop:equivalent2}, which may result in  an ineffective kernel and \eqref{Prop:equivalent4}, which permutes the isotropy groups (see the proof of Theorem \ref{thm:main:sing} for a precise statement in this case). The effectiveness of the action has an impact on the topological properties such as the cohomology ring  of $\EO$ (see Section \ref{sec:cohom}).
\end{remark}

As we will see in Section~\ref{sec:singset},  the equivalent representation \eqref{EQ:Simpler_reparametrization} for $\EO$ is crucial to the proof of Theorem~\ref{thm:main:sing}. From \eqref{Eq:order_isotropy} it easily follows that  
\begin{lemma}\label{L:Change_of_Parameteres}
   For the representation in \eqref{EQ:Simpler_reparametrization}, 
 the orders of the isotropy groups at $T_\sigma$, up to sign,  are given by 
 \begin{align}
    l_{Id} &= -(d+e)(a_2-b_2)-da_1 \label{eq:Id}\\
        	l_{(23)} &=  -(d+e)(a_2 - b_3) - da_1 \label{eq:(23)}\\
         l_{(132)} &= ca_2 + (c+e)(a_1 - b_3)\label{eq:132}\\
		l_{(12)} &= ca_2 + (c+e)(a_1-b_2) \label{eq:12} \\
		l_{(123)} &= c(a_2-b_3) - d (a_1 - b_2) \label{eq:123}\\
		l_{(13)} &= c(a_2-b_2) -d (a_1-b_3). \label{eq:13}
    \end{align}
\end{lemma}

The following lemma, which identifies the structure of the singular set in certain special cases,  plays a key role in the proof of Theorem~\ref{thm:main:sing}. In fact, it generalizes \cite[Lemma 5.3.11]{Y14}, and   since the same proof applies, we will omit  one here.
\begin{lemma}\label{L:b_2=b_3}
Let $\EO$ be an Eschenburg orbifold. If with $q_i = q_j$ and $b_i = b_j$ or $p_i = p_j$ and $a_i = a_j$ for some $i,j = 1,2,3$, i.e. $\EO$ admits a cohomogeneity $2$ action by some $S^1 \cdot \SU(2)$. Then the singular set  only consists of  smooth totally geodesic $2$-Spheres.
\end{lemma}

\section{  A Geometric Approach to the Positive Sectional Curvature Conditions on $\EO$}\label{sec:poscurv}
In this section we revisit the positive sectional curvature property for the Eschenburg $6$-orbifolds $\EO$ as stated in \cite{Y14}. In Subsection~\ref{subsec:poscurv} we provide a simple geometric interpretation of this condition, which forms  the foundation for our strategy to prove  Theorem~\ref{thm:main:sing}. This strategy is facilitated by some basic yet effective tools from convex geometry, with the details of their application explained in Subsection~\ref{SS:Separating_Hyperplane}.

\subsection{Revisiting positive  sectional curvature of $\EO$}\label{subsec:poscurv}
  We endow $\SU(3)$ with a Cheeger deformation of the bi-invariant metric along one of the subgroups $\U(2)_{ii}$ for $i = 1,2,3$ and equip $\ESO$ and $\EO$ with the corresponding submersion metrics turning them into Riemannian orbifolds. On $\ESO$ these metrics were originally found and studied by Eschenburg in \cite{Esch82} and provided a rich class of examples of non-negatively and positively curved manifolds (see \cite[Sections~2, 3]{Kerin2011} for a detailed construction of  Eschenburg's metrics and their curvature properties). 

There exists an Eschenburg metric with positive sectional curvature on  $\ESO$ if and only if 
\begin{align*}
    q_1,q_2,q_3 \notin[\min p_j, \max p_j].
\end{align*}
Note that switching the roles of $p$ and $q$ will induce an orbifold diffeomorphism between $E^7_{p,q}$ and $ E^7_{q,p}$. It might happen  that $E^7_{p,q}$ does not admit a positively curved Eschenburg metric, but $E^7_{q,p}$ does. \\

We now turn our attention to  the Eschenburg $6$-orbifolds $\EO$  which admit an Eschenburg metric with positive sectional curvature. In fact, every such orbifold is the base space of an orbifold fibration with a positively curved Eschenburg $7$-orbifold as a total space:
 \begin{theorem}[\cite{Y14}, Theorem 5.3.6] \label{thm:poscurv:yeroshkin}
    $\EO$ admits an Eschenburg metric with positive sectional curvature  if and only if there exist parameters $u,v,x,y \in \bb Z^3$, such that $E^7_{u,v}$ admits an Eschenburg metric with positive sectional curvature and $\EO \cong E^7_{u,v}/\Sph^1_{x,y}$. Moreover, the Eschenburg metric on $E^7_{u,v}$ is induced by the same metric on $\SU(3)$ as the metric of $\EO$.
\end{theorem}
Note, however, that if $\EO$ admits a metric with positive curvature,  it is possible that  neither $\ESO$ nor  $\rm E^7_{a, b}$ admits positive curvature (see \cite[Example~5.3.1.]{Y14}). Therefore, given an Eschenburg orbifold $\EO$, it is not practical to consider the orbifold fibrations from some $7$-dimensional Eschenburg orbifolds over them to determine whether they admit positive curvature. Instead this property can be checked using the following theorem:
\begin{proposition}[\cite{Y14}, Proposition 5.3.13]\label{prop:poscurv}
   Equip $\EO$ with a metric induced by a Cheeger deformation along $\U(2)_{33}$. Then $\EO$ is positively curved, if and only if for each $t \in [0,1]$ and each triple $(\eta_1, \eta_2,\eta_3)$ with $\sum_i \eta_i = 1$ we have both 
   \begin{align*}
        (1-t)b_1+tb_2 \neq \sum \eta_i a_i \quad  \text{ or } \quad (1-t)q_1+tq_2 \neq \sum \eta_i p_i,
   \end{align*}
   and 
   \begin{align*}
        b_3 \neq \sum \eta_i a_i \quad  \text{ or } \quad q_3 \neq \sum \eta_i p_i.
   \end{align*}    
\end{proposition}
In the same way as for the Eschenburg $7$-orbifolds, it is possible that the Eschenburg metric on $\EO$ does not have positive sectional curvature, but the Eschenburg metric on the equivalent orbifold $\mathcal{O}_{q, p}^{b, a}$ does. \\

After reviewing the known results regarding the positive curvature conditions on 
$\EO$, we now offer a geometric reformulation that is elegant in its simplicity.

To be more precise, we introduce the following  points in the Euclidean  plane: 
\begin{align*}
    P_i = \begin{pmatrix}
    p_i\\a_i
\end{pmatrix} \text{ and } Q_i = \begin{pmatrix}
    q_i\\b_i
\end{pmatrix} \text{ for } i = 1,2,3.
\end{align*}
Note that $\sum P_i = \sum Q_i$, since $\sum p_i = \sum q_i$ and $\sum a_i = \sum b_i$. Since permuting the parameters of $\EO$ is an equivalent change, $\EO$ has an Eschenburg metric with positive curvature if and only if there is a permutation $\sigma \in S_3$, such that $\mathcal{O}_{p,q_\sigma}^{a,b_\sigma}$ has positive curvature with an Eschenburg metric coming from a Cheeger deformation along $\U(2)_{33}$. From this we get the following reformulation of Proposition \ref{prop:poscurv}:
\begin{proposition} \label{prop:pos_curv_condition}
	 $\EO$ admits an Eschenburg metric with positive sectional curvature if and only if  there exists a permutation $\sigma$, such that for each $t \in \lbrack 0,1 \rbrack$ and each $\eta_1,\eta_2,\eta_3\ge 0$ with $\sum \eta_i = 1$ we have both
	\begin{align}
	(1-t) \cdot Q_{\sigma(1)}+ t\cdot Q_{\sigma(2)}\neq \sum_i \eta_i P_i
	\end{align}
 and
    \begin{align}
	Q_{\sigma(3)} \neq \sum_i \eta_i P_i.\label{lem:pos_curv_condition_2}
	\end{align}
\end{proposition}

We will now explain the geometric interpretation of Proposition~\ref{prop:pos_curv_condition}: We denote the triangles in the Euclidean plane spanned by the points $P_i$,   respectively $Q_i$, by $\Delta_P$, respectively $\Delta_Q$, and a line segment between two points $X,Y \in \R^2$ by $L(X,Y)$. 
Then $\EO$ admits an Eschenburg metric with positive sectional curvature if and only if there exists $\sigma \in S_3$  such that $L(Q_{\sigma(1)},Q_{\sigma(2)}) \cup \lbrace Q_{\sigma(3)} \rbrace$ does not intersect the triangle $\Delta_P$. We illustrate this condition  in Figure~\ref{fig:triangles}: In the first case the corresponding orbifold cannot have an Eschenburg metric with positive sectional curvature since any edge of one of the triangles $\Delta_Q$ and $\Delta_P$ intersects the other triangle. In the second case, the corresponding Eschenburg orbifold admits a metric with positive sectional curvature, since the line segment connecting $Q_1$ and $Q_2$ does not intersect the triangle $\Delta_P$.
Note that   since $\sum P_i = \sum Q_i$, the triangles $\Delta_P$ and $\Delta_Q$ have the same centroid $c$. Therefore,  it can easily be shown (see Lemma \ref{P:Same_Centroid}) that the midpoint of $Q_1$ and $Q_2$ is contained in $\Delta_P$, if $Q_3 \in \Delta_P$. Hence, Condition~\ref{lem:pos_curv_condition_2} is not necessary.  
   \definecolor{ffqqqq}{rgb}{1,0,0}
\definecolor{wwzzqq}{rgb}{0,0.6,0}
\begin{figure}[h!]\caption{Sectional curvatures and triangles}\label{fig:triangles}
\begin{tikzpicture}[line cap=round,line join=round,>=triangle 45,x=1cm,y=1cm]
\clip(-9,-4.5) rectangle (10,3.5);
\fill[line width=2pt,color=wwzzqq,fill=wwzzqq,fill opacity=0.27] (-6,1) -- (-5,-3) -- (-2,1) -- cycle;
\fill[line width=2pt,fill=black,fill opacity=0.12] (-4,2) -- (-3,-2) -- (-6,-1) -- cycle;
\fill[line width=2pt,color=wwzzqq,fill=wwzzqq,fill opacity=0.27] (2,3) -- (3,-3) -- (5,0) -- cycle;
\fill[line width=2pt,fill=black,fill opacity=0.12] (1,1) -- (2,-1) -- (7,0) -- cycle;
\draw [line width=2pt,color=wwzzqq] (-6,1)-- (-5,-3);
\draw [line width=2pt,color=wwzzqq] (-5,-3)-- (-2,1);
\draw [line width=2pt,color=wwzzqq] (-2,1)-- (-6,1);
\draw [line width=2pt] (-4,2)-- (-3,-2);
\draw [line width=2pt] (-3,-2)-- (-6,-1);
\draw [line width=2pt] (-6,-1)-- (-4,2);
\draw [line width=2pt,color=wwzzqq] (2,3)-- (3,-3);
\draw [line width=2pt,color=wwzzqq] (3,-3)-- (5,0);
\draw [line width=2pt,color=wwzzqq] (5,0)-- (2,3);
\draw [line width=2pt] (1,1)-- (7,0);
\draw [line width=2pt,color=ffqqqq] (2,-1)-- (1,1);
\draw [line width=2pt] (7,0)-- (2,-1);
\draw (-7.5,-3.75) node[anchor=north west] {(1) No positive sectional curvature};
\draw (0.5,-3.75) node[anchor=north west] {(2) Positive sectional curvature};
\draw[fill=black] (-4.3333333,-0.3333333) circle (1.5pt);
\draw[color = black] (-4.2,-0.5) node {$c$};
\draw [fill=wwzzqq] (-6,1) circle (2.5pt);
\draw[color=wwzzqq] (-5.809151686768261,1.3) node {$P_1$};
\draw [fill=wwzzqq] (-5,-3) circle (2.5pt);
\draw[color=wwzzqq] (-5,-3.35) node {$P_3$};
\draw [fill=wwzzqq] (-2,1) circle (2.5pt);
\draw[color=wwzzqq] (-1.8029460044764865,1.3) node {$P_2$};
\draw [fill=black] (-4,2) circle (2.5pt);
\draw[color=black] (-3.7973016716435706,2.3) node {$Q_1$};
\draw [fill=black] (-3,-2) circle (2.5pt);
\draw[color=black] (-2.6,-2) node {$Q_2$};
\draw [fill=black] (-6,-1) circle (2.5pt);
\draw[color=black] (-6.4,-1) node {$Q_3$};
\draw[fill=black] (3.3333333,0) circle (1.5pt); 
\draw[color = black] (3.46666,-0.16666) node {$c$};
\draw [fill=wwzzqq] (2,3) circle (2.5pt);
\draw[color=wwzzqq] (2,3.3) node {$P_1$};
\draw [fill=wwzzqq] (3,-3) circle (2.5pt);
\draw[color=wwzzqq] (3,-3.35) node {$P_3$};
\draw [fill=wwzzqq] (5,0) circle (2.5pt);
\draw[color=wwzzqq] (5.3,0) node {$P_2$};
\draw [fill=black] (7,0) circle (2.5pt);
\draw[color=black] (7.35,0) node {$Q_3$};
\draw [fill=black] (1,1) circle (2.5pt);
\draw[color=black] (1,1.3) node {$Q_1$};
\draw [fill=black] (2,-1) circle (2.5pt);
\draw[color=black] (2,-1.35) node {$Q_2$};
\draw[color=ffqqqq] (0.75,-0.5) node {$L(Q_1,Q_2)$};
\end{tikzpicture}
\end{figure}
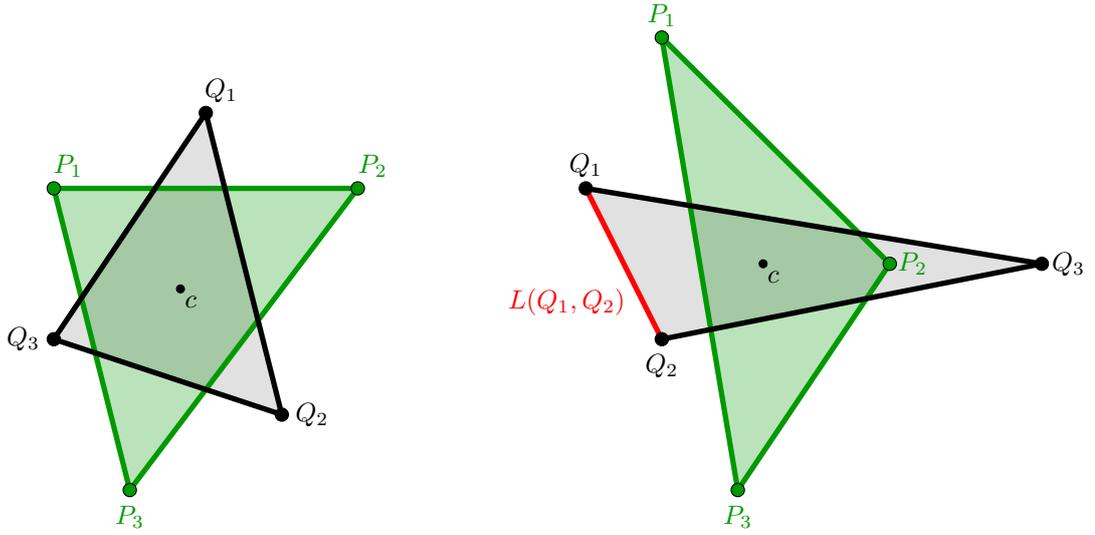

\begin{remark}\label{rem:poscurv}\mbox{}
    \begin{enumerate}
        \item The positive curvature condition for a $7$-dimensional Eschenburg orbifold $E^7_{p,q}= \SU(3)\sslash \Sph^1_{p,q}$ with a Cheeger deformation along $\U(2)_{33}$ is as follows: For all $t \in \lbrack 0,1\rbrack $
    \begin{align*}
        (1-t)q_1 + tq_2, q_3 \notin \lbrack \min p_i, \max p_i\rbrack
    \end{align*}
    Since $\lbrack \min p_i, \max p_i\rbrack$ is given by the set of convex combinations of $p_1,p_2,p_3$, this is analogous to Proposition~\ref{prop:pos_curv_condition}.
    \item Operations \eqref{Prop:equivalent2}, \eqref{Prop:equivalent3} and \eqref{Prop:equivalent5} of Proposition \ref{Prop:equivalent} do not affect the positive curvature condition  since \eqref{Prop:equivalent2} is a scaling and \eqref{Prop:equivalent3} a translation of the $P_i$ and $Q_i$. Operation~\eqref{Prop:equivalent5} corresponds to a linear change of the coordinate system and hence cannot affect the condition. Operation~\eqref{Prop:equivalent1} affects the positive curvature condition in the sense  that the roles of $P_i$ and $Q_i$ are interchanged.
    \item  \label{rem:poscurv_3} Note that $\lbrace P_1,P_2, P_3\rbrace \cap \lbrace Q_1, Q_2, Q_2\rbrace = \emptyset$, since otherwise one of $N_\sigma =0$. To see this, assume without loss of generality, that $P_1 = Q_1$. Therefore we have $p_1 = q_1 $, $a_1 = b_1$ and hence
    \begin{align*}
        N_{\Id} = \vert (p_1-q_1)(a_2 - b_2)- (a_1 - b_1)(p_2- q_2)\vert = 0.
    \end{align*}
This violates the assumption  that the action is almost free.
    \end{enumerate}
\end{remark}
\subsection{Separating Axis Theorem and positive curvature}\label{SS:Separating_Hyperplane}
In this subsection, we recall some elementary facts from convex geometry and prove some lemmata which will play an important role in developing a strategy to prove Theorem~\ref{thm:main:sing} (see \cite{Sch14} for more details). 

For $u\in \mathbb{R}^n\setminus \{0\}$ and $\alpha\in \R$ define the hyperplane $H_{u,\alpha}$  as $$H_{u, \alpha}:=\{x\in \R^n\mid \langle x, u\rangle=\alpha\}.$$
The hyperplane $H_{u,\alpha}$ bounds the two half-spaces 
$$H^-_{u, \alpha}=\{x\in \R^n\mid \langle x, u\rangle \leq \alpha\}, \qquad H^+_{u, \alpha}=\{x\in \R^n\mid \langle x, u\rangle \geq \alpha\}.$$

 Let $K_1, K_2\subseteq \R^n$ be two subsets and $H_{u, \alpha}$ be a hyperplane. The
hyperplane $H_{u, \alpha}$ separates $K_1$ and $K_2$ if, $K_1\subseteq H^-_{u, \alpha}$ and $K_2\subseteq H^+_{u, \alpha}$, or vice versa. The subsets $K_1, K_2$ are said to be strongly separated by $H_{u, \alpha}$ if there is a number $\varepsilon>0$ such that $H_{u, \alpha-\varepsilon}$ and $H_{u, \alpha+\varepsilon}$ 
both separate $K_1$ and $K_2$. 
\begin{theorem}\label{Thm:Separation}\cite[Theorem~1.3.7]{Sch14}
    Let  $K_1, K_2\subseteq \R^n$ be non-empty convex subsets with $K_1\cap K_2=\emptyset$. Then $K_1$ and $K_2$ can be separated. If $K_1$ is compact and $K_2$ is closed, then they can be strongly separated. 
\end{theorem}

The following theorem must be known in convex geometry, but we were unable to find a proof in the literature and hence provide one for completeness. 
\begin{theorem}\label{THM_Separation_Axis}
Let $K_1, K_2\subseteq \R^2$ be two polygons with $K_1\cap K_2= \emptyset$. Then at least one edge of one of the polygons can be extended to a separating line $H_{v, \alpha}$ such that  $v$ is an outer normal vector to the corresponding polygon. 
\end{theorem}
\begin{proof}
    Let $K=K_1-K_2=\{x-y\mid x\in K_1, y\in K_2\}$ be the Minkowski sum of $K_1$ and   $-K_2$. Note  that $K$ is a compact convex polygon as well.  Furthermore, since $K_1\cap K_2=\emptyset$, then $0\notin K$. Hence, $0$ and $K$ can be (strongly) separated. We show that we can assume  the separating line $L$ is parallel to  one of the edges of $K$. 

    By the proof of \cite[Theorem 2.4.3]{Sch14}, every polygon is the intersection of finitely many closed half-spaces $H^-_k$ bounded by a support Line $H_k$. Moreover, every edge $E_i$ of the polygon  satisfies $E_i=K\cap H_i$, i.e.  every edge can be extended to a support line. Now since $0\notin K$, there exists $i$, such that $0\notin H_i^-$. Hence $H_i=H(v, \alpha)$ is a separating line. Then it follows from the proof of \cite[Lemma 1.3.6]{Sch14} that a line  $L$ parallel to $H_i$ (strongly) separates $K_1$ and  $K_2$.  
   Moreover,  since the edges of the Minkowski sum $K$ are parallel to the edges of the original polygons, we deduce that $L$ is parallel to one of the edges of one of the polygons $K_1$ and $K_2$. Note that we can always choose $v$ in such a way that $v$ is an outer normal to the corresponding polygon, that is, the polygon for which $L$ is parallel to an edge.  This finishes the proof.
\end{proof}

\begin{remark}[Separating Axis]\label{REM:Separating_Axis}
  Let $K_1$ and $K_2$ be two disjoint polygons. Let $L$ be a separating line, which we assume by Theorem~\ref{THM_Separation_Axis}  to be an edge of one of the polygons. Let $L^\perp$ be a line orthogonal to $L$. Then the projections of $K_1$ and $K_2$ to $L^\perp$ are disjoint. We call  $L^\perp$  a  \emph{separating axis}.  Note that $L^\perp$ is not unique, but we may require it to pass through a particular point.
\end{remark}

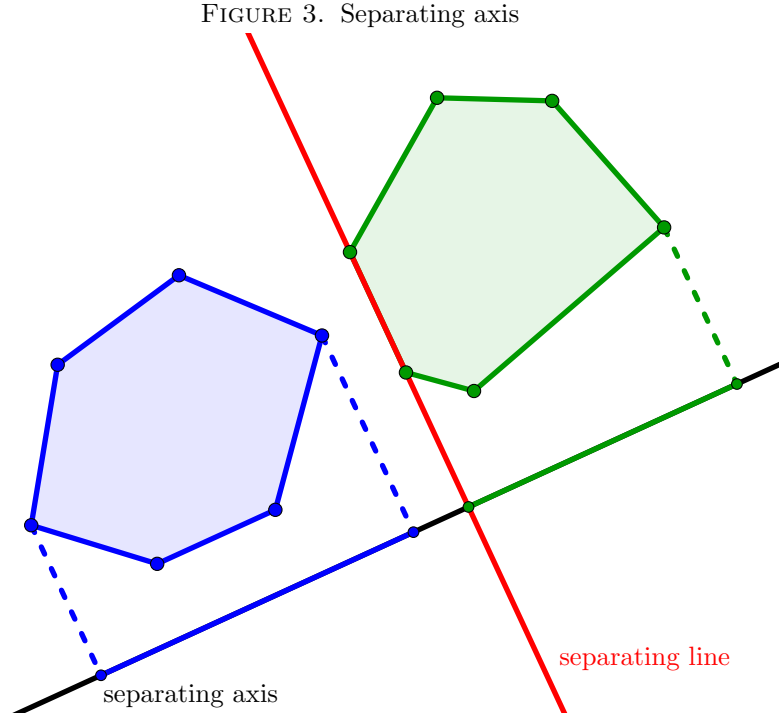
\begin{figure}[h!]
 \caption{Separating axis}
    \label{fig:separatingaxis}
    \centering
    \definecolor{ffqqqq}{rgb}{1,0,0}
\definecolor{qqzzqq}{rgb}{0,0.6,0}
\definecolor{qqqqff}{rgb}{0,0,1}
\begin{tikzpicture}[line cap=round,line join=round,>=triangle 45,x=1cm,y=1cm]
\clip(-7,-5) rectangle (4,4);
\draw (-5,-4.5) node[anchor=north west] {separating axis};
\draw (1,-4) node[anchor=north west, color = ffqqqq] {separating line};
\fill[line width=2pt,color=qqqqff,fill=qqqqff,fill opacity=0.1] (-3.8821863119301816,0.798719275486731) -- (-5.478100189747409,-0.3879859157106935) -- (-5.825927573374241,-2.5158710861336617) -- (-4.168632392564043,-3.027381944408414) -- (-2.613639383408796,-2.311266742823761) -- (-2,0) -- cycle;
\fill[line width=2pt,color=qqzzqq,fill=qqzzqq,fill opacity=0.1] (-0.48575421298582583,3.15166922355059) -- (-1.6315385355212713,1.1056257904515823) -- (-0.8949628996056278,-0.4902880873656439) -- (0,-0.7358132993375248) -- (2.5014691993387284,1.4329927397474236) -- (1.0283179275074412,3.11074835488861) -- cycle;
\draw [line width=2pt,color=qqqqff] (-3.8821863119301816,0.798719275486731)-- (-5.478100189747409,-0.3879859157106935);
\draw [line width=2pt,color=qqqqff] (-5.478100189747409,-0.3879859157106935)-- (-5.825927573374241,-2.5158710861336617);
\draw [line width=2pt,color=qqqqff] (-5.825927573374241,-2.5158710861336617)-- (-4.168632392564043,-3.027381944408414);
\draw [line width=2pt,color=qqqqff] (-4.168632392564043,-3.027381944408414)-- (-2.613639383408796,-2.311266742823761);
\draw [line width=2pt,color=qqqqff] (-2.613639383408796,-2.311266742823761)-- (-2,0);
\draw [line width=2pt,color=qqqqff] (-2,0)-- (-3.8821863119301816,0.798719275486731);
\draw [line width=2pt,color=qqzzqq] (-0.48575421298582583,3.15166922355059)-- (-1.6315385355212713,1.1056257904515823);
\draw [line width=2pt,color=qqzzqq] (-1.6315385355212713,1.1056257904515823)-- (-0.8949628996056278,-0.4902880873656439);
\draw [line width=2pt,color=qqzzqq] (-0.8949628996056278,-0.4902880873656439)-- (0,-0.7358132993375248);
\draw [line width=2pt,color=qqzzqq] (0,-0.7358132993375248)-- (2.5014691993387284,1.4329927397474236);
\draw [line width=2pt,color=qqzzqq] (2.5014691993387284,1.4329927397474236)-- (1.0283179275074412,3.11074835488861);
\draw [line width=2pt,color=qqzzqq] (1.0283179275074412,3.11074835488861)-- (-0.48575421298582583,3.15166922355059);
\draw [line width=2pt,color=ffqqqq,domain=-10.777352681473843:12.056492031911102] plot(\x,{(-1.78941797134538-1.5959138778172262*\x)/0.7365756359156435});
\draw [line width=2pt,domain=-10.777352681473843:12.056492031911102] plot(\x,{(-3.574847086310587--0.7365756359156435*\x)/1.5959138778172262});
\draw [line width=2pt,dash pattern=on  3pt off 6pt,color=qqqqff] (-5.825927573374241,-2.5158710861336617)-- (-4.907803437164254,-4.505140047921966);
\draw [line width=2pt,dash pattern=on  3pt off 6pt,color=qqqqff] (-2,0)-- (-0.7964878048780477,-2.607609756097561);
\draw [line width=2pt,dash pattern=on  3pt off 6pt,color=qqzzqq] (2.5014691993387284,1.4329927397474236)-- (3.459715748236801,-0.6432081161983985);
\draw [line width=2pt,color=qqqqff] (-4.907803437164254,-4.505140047921966)-- (-0.7964878048780477,-2.607609756097561);
\draw [line width=2pt,color=qqzzqq] (-0.07205463828229842,-2.273255986899523)-- (3.459715748236801,-0.6432081161983985);
\begin{scriptsize}
\draw [fill=qqqqff] (-3.8821863119301816,0.798719275486731) circle (2.5pt);
\draw [fill=qqqqff] (-5.478100189747409,-0.3879859157106935) circle (2.5pt);
\draw [fill=qqqqff] (-5.825927573374241,-2.5158710861336617) circle (2.5pt);
\draw [fill=qqqqff] (-4.168632392564043,-3.027381944408414) circle (2.5pt);
\draw [fill=qqqqff] (-2.613639383408796,-2.311266742823761) circle (2.5pt);
\draw [fill=qqqqff] (-2,0) circle (2.5pt);
\draw [fill=qqzzqq] (-0.48575421298582583,3.15166922355059) circle (2.5pt);
\draw [fill=qqzzqq] (-1.6315385355212713,1.1056257904515823) circle (2.5pt);
\draw [fill=qqzzqq] (-0.8949628996056278,-0.4902880873656439) circle (2.5pt);
\draw [fill=qqzzqq] (0,-0.7358132993375248) circle (2.5pt);
\draw [fill=qqzzqq] (2.5014691993387284,1.4329927397474236) circle (2.5pt);
\draw [fill=qqzzqq] (1.0283179275074412,3.11074835488861) circle (2.5pt);
\draw[color=black] (-10.531827469501962,-6.700029906821133) node {$n$};
\draw [fill=qqqqff] (-4.907803437164254,-4.505140047921966) circle (2pt);
\draw [fill=qqqqff] (-0.7964878048780477,-2.607609756097561) circle (2pt);
\draw [fill=qqzzqq] (-0.07205463828229842,-2.273255986899523) circle (2pt);
\draw [fill=qqzzqq] (3.459715748236801,-0.6432081161983985) circle (2pt);

\end{scriptsize}
\end{tikzpicture}
\end{figure}

Let $H$ be a closed non-empty convex subset of $\R^n$. Recall that for each $x\in \R^n$, there exists a unique point $\pr_H(x)\in H$ satisfying 
$$\lvert x-\pr_H(x)\rvert\leq \lvert x-y\rvert,$$
for all $y\in H$ (see for example \cite[p. 9]{Sch14}). Then we can define a map $\pr_H:\mathbb{R}^n\to H$, called the \emph{metric projection} of $H$.  Note that $\pr_H$ is the orthogonal projection, if $H$ is a linear subspace of $\R^n$.

We will now justify that only Condition \ref{lem:pos_curv_condition_2} of Proposition~\ref{prop:pos_curv_condition} is needed as already stated in the discussion right before Remark~\ref{rem:poscurv}.
\begin{lemma} \label{P:Same_Centroid}
    Let $(p_1,\ldots, p_k)$ and $(q_1, \ldots, q_k)$ be two collections of points in $\R^n$ with the same centroid. If  $q_k \in \Conv \lbrace p_1,\ldots, p_k\rbrace $, then $\Conv \lbrace p_1, \ldots , p_k \rbrace$ contains the centroid of $(q_1,\ldots, q_{k-1})$.
\end{lemma}
\begin{proof}
    We first proof the statement for $n = 1$. In this case $\Conv\lbrace p_1, \ldots , p_k\rbrace = [\min p_i, \max p_i]$. Without loss of generality we assume $\min p_i = p_1$ and $\max p_i = p_k$. Since both collections have the same centroid we have
    \begin{align*}
        &\sum_{i = 1}^{k-1} q_i + p_1 \le \sum_{i=1}^k q_i = \sum_{i=1}^k  p_i \le p_1 + (k-1)p_k\\
        \Rightarrow \quad & {1 \over k-1}\sum_{i = 1}^{k-1} q_i \le p_k
    \end{align*}
    and analogously    ${1 \over k-1}\sum_{i = 1}^{k-1} q_i \ge p_1$.   Therefore ${1 \over k-1}\sum_{i = 1}^{k-1} q_i  \in\Conv \lbrace p_1, \ldots, p_k\rbrace$. \\
    For arbitrary $n$, let ${1 \over k-1}\sum_{i = 1}^{k-1} q_i  \notin\Conv \lbrace p_1, \ldots, p_k\rbrace$. By Theorem \ref{Thm:Separation} there exists a Hyperplane $H$ strongly separating ${1 \over k-1}\sum_{i = 1}^{k-1} q_i$ and $\Conv \lbrace p_1, \ldots, p_k\rbrace$. Let $L$ be a line perpendicular to $H$ through the origin and denote by $\pr \colon \R^n \to L$ the orthogonal projection. Since $\pr$ is linear, we have $\pr(\Conv \lbrace p_1, \ldots, p_k\rbrace) = \Conv \lbrace \pr(p_1), \ldots, \pr(p_k)\rbrace$. Now the collections $(\pr(p_1),\ldots, \pr(p_k))$ and $(\pr(q_1),\ldots,\pr(q_k))$ fulfill the assumptions of the lemma in a one dimensional space, but $\pr(q_k) \notin  \Conv \lbrace \pr(p_1), \ldots, \pr(p_k)\rbrace$, since they are separated by $H$, a contradiction.
\end{proof}

  From now on we denote the interval spanned by the projection of $\Delta_P$ along a line  by $J_P$ without further reference. Since in our case the  triangles $\Delta_P$ and $\Delta_Q$ have the same centroids it follows immediately from Lemma~\ref{P:Same_Centroid} that Condition~\ref{lem:pos_curv_condition_2} is not necessary:
\begin{corollary}\label{C:Sec>0_Geometric}
    The orbifold $\mathcal{O}^{a, b}_{p, q}$ has positive curvature if and only if one of the following equivalent conditions hold 
    \begin{enumerate}
        \item At least one of the edges  $L(Q_i, Q_j)$ of $\Delta_Q$, $i\neq j$, does not intersect $\Delta_P$. 
        \item There exists a line $L$ perpendicular to one of  the edges of one of the triangles such that the projections of no vertices of $\Delta_Q$ along this line lie in $J_P$.
    \end{enumerate}
\end{corollary}

We can use the Separating Axis Theorem to obtain the following strengthening of Theorem \ref{thm:poscurv:yeroshkin}, which will be the main tool in the Proof of Theorem~\ref{thm:main:sing}~\eqref{thm:main:sing:same_Parity}.
\begin{theorem} \label{thm:pos:curv:fibration}
    Let $\EO$ be an Eschenburg orbifold with positive sectional curvature, then there exists and equivalent Eschenburg orbifold $\mathcal{O}_{p^\prime,q^\prime}^{a^\prime,b^\prime}$, such that $E^7_{p^\prime,q^\prime}$ is of cohomogeneity $2$ and admits positive sectional curvature. Furthermore the metrics on $\mathcal{O}_{p^\prime,q^\prime}^{a^\prime,b^\prime}$ and $E_{p^\prime, q^\prime}$ are  induced by the same metric on $\SU(3)$ as the metric on $\EO$.
\end{theorem}

\begin{proof}
    Let $\EO$ be an Eschenburg $6$-orbifold with positive sectional curvature. By Corollary~\ref{C:Sec>0_Geometric} we can assume that one of $L(X,Y)$ for $X,Y \in \lbrace Q_1,Q_2,Q_3\rbrace$ or $X,Y \in \lbrace P_1,P_2,P_3\rbrace$ with $X \neq Y$ acts as a separating line. Let $g = \gcd(x_1-y_1,x_2-y_2)$ and denote  $v = {1\over g} (X-Y)=(v_1, v_2)$. There exist $m,n\in \bb Z$, such that $mv_1 +nv_1= 1$. We set
    \begin{align*}
        A = \begin{pmatrix}
            v_2 & -v_1\\ m& n
        \end{pmatrix} \text{ and } \begin{pmatrix}
            p^\prime \\ a^\prime
        \end{pmatrix} \coloneqq A \cdot \begin{pmatrix}
            p \\a
        \end{pmatrix}, \quad \begin{pmatrix}
            q^\prime \\ b^\prime
        \end{pmatrix} \coloneqq A \cdot \begin{pmatrix}
            q \\b
        \end{pmatrix}
    \end{align*}
    $\mathcal{O}_{p^\prime, q^\prime}^{a^\prime, b^ \prime}$ has positive sectional curvature and the line $L(X^\prime, Y^\prime)$ with $X^\prime = AX$ and $Y^\prime = AY$ is separating. Since $Av = (0, 1)^T$, we have that $L(X^\prime, Y^\prime) = 0 \times \R$ is separating. Therefore $\R \times 0$ is a separating axis, and hence $E^7_{p^\prime,q^\prime}$ admits positive sectional curvature. Furthermore, either two of the $p^\prime_i$'s or two of the $q^\prime_i$'s coincide. Therefore $E^7_{p^\prime,q^\prime}$ is actually a cohomogeneity two Eschenburg orbifold with positive sectional curvature.
\end{proof}

We conclude this section with the following lemma  which asserts that in the special case of Lemma~\ref{L:b_2=b_3}, we have positive sectional curvature up to switching the parameters:
\begin{corollary}\label{C:Degenerate_Triangle}
  
   Suppose that in one of the triangles $\Delta_P$ or $\Delta_Q$ at least two vertices coincide, i.e. $p_i = p_j$ and $a_i = a_j$ or $q_i = q_j $ and $b_i = b_j$ for some $i\neq j$, or equivalently $\EO$ admits a cohomogeneity two action by some $S^1 \cdot \SU(2)$. Then either $\EO$ or $\mathcal{O}^{b, a}_{q, p}$ has positive sectional curvature.
\end{corollary}

\begin{proof}
        Assume that in $\Delta_Q$ at least two vertices coincide. Thus $\Delta_Q$ is a segment in $\R^2$. If either of the endpoints  lie in $\Delta_P$, then $\EO$ has  positive sectional curvature. Now, consider the case where one endpoint, say $Q_i$, intersects $\Delta_P$. Then we have two scenarios: either $Q_i=Q_j$, for some $j\neq i$, or $Q_i\neq Q_j$, for $j\neq i$. If $Q_i\neq Q_j$, for $j\neq i$, then by Lemma~\ref{P:Same_Centroid}, and Part~\eqref{rem:poscurv_3} of Remark~\ref{rem:poscurv}, the other endpoint intersects $\Delta_P$ as well, but not at a vertex. By the properties of the plane, this implies that one edge, including the vertices, of the triangle $\Delta_P$ does not intersect with $\Delta_Q$. That is, the orbifold $\mathcal{O}^{b, a}_{q, p}$ has positive curvature. Now assume that $Q_i=Q_j$, for some $j\neq i$. If the other endpoint intersects  $\Delta_P$, similarly, $\mathcal{O}^{b, a}_{q, p}$ has positive curvature. Let us assume that the other endpoint $Q_k$, $k\neq i, j$ does not lie within $\Delta_P$. We claim that the segment, $\Delta_Q$,  does not pass through a vertex of $\Delta_P$. Assume it does, then since the midpoint of the segment and the centroid of $\Delta_P$ coincide, the segment necessarily lies on one of the median of the $\Delta_P$ and leaves the triangle by the "$2/3$ property" of the medians, a contradiction. Therefore, the segment does not pass through the vertices and hence does not intersect one of the edges, including the vertices, of $\Delta_P$. Whence, $\mathcal{O}^{b, a}_{q, p}$ has positive curvature. 
\end{proof}

\section{Proof of Theorem~\ref{thm:main:sing}}\label{sec:singset}
In this section, we prove Theorem~\ref{thm:main:sing}.  
We divide the proof into two subsections. In Subsection~\ref{SS:Part_1}, we  first prove  Part~\ref{thm:main:sing:same_Parity}  for isometric  almost free $\S^1$ actions on positively curved cohomogeneity $2$ Eschenburg $7$-orbifolds. Then by Theorem~\ref{thm:pos:curv:fibration}, we arrive at the same conclusion for the general case. While we could follow the same approach to prove Part~\ref{thm:main:sing:smoothsphere}, we prefer to present a direct proof, as it is relatively less convoluted and provides additional information  needed later. This will be done in Subsection~\ref{SS:Part_2}. The main tool is the \emph{Separating Axis Theorem}, which allows us to reduce the positive curvature conditions to a single condition (Corollary~\ref{C:Sec>0_Geometric}) and to develop a geometric strategy to verify this condition.

\subsection{Proof of Theorem \ref{thm:main:sing}, Part~\ref{thm:main:sing:same_Parity}}\label{SS:Part_1} The main objective of this section is to prove the following proposition. From this, as detailed below, we can easily deduce Part~\ref{thm:main:sing:same_Parity} of Theorem \ref{thm:main:sing}. 
\begin{proposition}\label{Prop:Cohom_2}
    Let $\ESO$ be a  cohomogeneity $2$ Eschenburg orbifold 
    with the  almost free  biquotient action of $\S^1_{a, b}$. If  three of $\ast_{\sigma}$'s  with the same parity are regular and either $\ESO$  or $\mathrm{E}^7_{q, p}$ admits positive curvature, then  $\T$  acts freely. 
\end{proposition}
Now, we show how one  proves  Theorem \ref{thm:main:sing}, Part~\ref{thm:main:sing:same_Parity}:

\begin{proof}([Theorem \ref{thm:main:sing}, Part~\ref{thm:main:sing:same_Parity}])
Assume by contradiction that 
  all $\ast_\sigma$ contained in $\Sigma$ have the same parity. Then there are three regular $\ast_\sigma$'s with the same parity. By Theorem~\ref{thm:pos:curv:fibration}, there exists an equivalent Eschenburg orbifold $\mathcal{O}^{a^\prime, b^\prime}_{p^\prime,q^\prime}$ with the same singular set and such that $\mathrm{E}^7_{p^\prime,q^\prime}$ is a  positively curved  cohomogeneity $2$ Eschenburg orbifold.  
  By Proposition~\ref{Prop:Cohom_2}, such orbifold is a smooth manifold. This is a contradiction to $\Sigma \neq \emptyset$.
\end{proof}

\begin{proof}([Proposition~\ref{Prop:Cohom_2}])
 Let   $\tau, \tau^\prime \in \Sph_{3}$. Then  for the equivalent action of $S^1_{p_\tau,q_{\tau^\prime}} \times S^1_{a_\tau,b_{\tau^\prime}}$   we see that $N_\sigma = \tilde N_{\tau^\prime \circ \sigma \circ \tau}$, where $\tilde N_\sigma$ is the order of the isotropy group of the $S^1_{p_\tau,q_{\tau^\prime}} \times S^1_{a_\tau,b_{\tau^\prime}}$ action. By applying suitable reflections $\tau$ and $\tau^\prime$, we can assume
\begin{align*}
    p=(c,d,e), \quad q = (c+d+e,0,0)
\end{align*} without changing the parity of the $\ast_\sigma$'s.

Furthermore,  without loss of generality, we assume that the three regular points $\ast_{\sigma}$'s  have even parity, i.e. $N_{Id} = N_{(123)} = N_{(132)} =1$. More precisely, if we choose $\tau^\prime=Id$, $\tau=(12)$, then we have $\tilde N_{Id} =N_{(12)}$, $\tilde N_{(123)}=N_{(13)}$, $\tilde N_{(132)} =N_{(23)}$.\\
 We can assume that the parameters of the circle $\S^1_{a, b}$ are given by 
 $$a=(a_1, a_2, a_3), \quad b=(0, b_2, b_3).$$

Now, we proceed the proof by assuming that $N_{Id} = N_{(123)} = N_{(132)} =1$. Notice that this condition and the equations in Lemma~\ref{L:Change_of_Parameteres} imply that $\gcd(c, d)=\gcd(c, e)=\gcd(e, d)=1$, i.e.  $E^7_{p, q}$ is a smooth manifold. 

Observe that if $b_2=b_3$, then by equations in Lemma~\ref{L:Change_of_Parameteres}, all the isotropy groups are trivial. Thus $\mathcal{O}_{p,q}^{a,b}$ is a smooth manifold. From now on, throughout the proof, we assume that $b_2-b_3\neq 0$.

From Equations in Lemma~\ref{L:Change_of_Parameteres}  we get
\begin{align}
        l_{(132)}-l_{(12)} &= (c+e)\cdot(b_2 - b_3). \label{eq:132-12}
        \\ 
        l_{(123)} - l_{(13)} & = (c+d) \cdot (b_2 -b _3).\label{eq:123-13}
        \\
        l_{Id} - l_{(23)} & = (e+d) \cdot (b_2 -b _3).\label{eq:Id-23}
    \end{align}
 Since $\ESO$ or $\mathrm{E}^7_{q, p}$ has positive curvature, we have 
 \begin{align}\label{Eq:No_opposite_sign}
  0\neq  c\neq -d\neq 0, \quad 0\neq  c\neq -e\neq 0,\quad 0\neq  e\neq -d\neq 0.   
 \end{align}

 Therefore, we have:
 \begin{align*}
    l_{(132)}-l_{(12)}\neq 0,\quad   l_{(123)} - l_{(13)}\neq 0, \quad  l_{Id} - l_{(23)}\neq 0. 
\end{align*}
Let \begin{align*}
    x:=c+e, \quad y:=b_2-b_3, \quad l:=l_{(132)}, \quad t:=l_{\Id}, \quad r:=t-l_{(123)}.
\end{align*} 
To avoid treating similar cases and streamline the proof, we make the following additional choices.
By multiplying $p,q$, respectively $a, b$,   with $-1$, if necessary,    we can assume that $c+d+e \ge 0$, respectively $xy-l > 0$. 
After applying  these restrictions,  we need to consider the following four cases: $(l_{Id},l_{(123)})\in \lbrace (1,1), (1,-1), (-1, 1),(-1,-1)\rbrace$.  However,  if one of $l_\sigma =1$ for $\sigma \in \lbrace Id, (123) \rbrace$, we can always assume it to be $l_{Id}$. Otherwise, we apply $(13)$ to $p$ and $a$ and $(23)$ to $q$ and $b$ to obtain  the treated case. As a result, we only need to consider three cases: $(l_{Id},l_{(123)})\in \lbrace (1,1), (1,-1),(-1,-1)\rbrace$. In other words, whenever $t=1$, we have $r=0$, or $r=2$ and whenever $t=-1$, we only have $r=0$.

From Equation \eqref{eq:132} and since $x=c+e\neq 0$, we get \begin{align}
    a_1=\frac{l-ca_2}{x}+b_3=\frac{1-ca_2}{x}+b_2-y. \label{eq:Case2_a1}
\end{align}
 We have 
$$r=l_{Id}-l_{(123)}=-a_2(d+x)-cy+xb_2.$$ Then we get
\begin{align}\label{Eq:Case2_b2}
    b_2=\frac{r+a_2(d+x)+cy}{x}.
\end{align}
 From \eqref{eq:Id}, \eqref{Eq:Case2_b2}, and \eqref{eq:Case2_a1},   we have  $$xt=xl_{Id}=d(xy-l)+(x-c)(r+cy).$$
Note that $xy-l=-l_{(12)}\neq 0$. Whence we have $$d=\frac{xt+(c-x)(r+cy)}{xy-l}.$$

by substituting $d$ in \eqref{Eq:Case2_b2}, we get 
$$b_2:=\frac{r(xy-l)+a_2(xt+(c-x)(r+cy)+x^2y-x)+cy(xy-l)}{x(xy-l)}.$$
Now  we list all the parameters of the torus representation as follow:

 \begin{equation}
\begin{aligned}\label{Eq:xy-1}
 p &= (c,\frac{xt+(c-x)(r+cy)}{xy-l} ,x-c), \qquad q=(\frac{xt+(c-x)(r+cy)}{xy-l}+x, 0,0),\\ a_1 &= \frac{l-ca_2}{x}+\frac{r(xy-l)+a_2(xt+(c-x)(r+cy)+x^2y-x)+cy(xy-l)}{x(xy-l)}-y,\\
 a_3&=\frac{r(xy-l)+a_2(xt+(c-x)(r+cy)+x^2y-x)+cy(xy-l)}{x(xy-l)}-\frac{l-ca_2}{x}-a_2,\\
 b_1&=0,\\
 b_2&=\frac{r(xy-l)+a_2(xt+(c-x)(r+cy)+x^2y-x)+cy(xy-l)}{x(xy-l)},\\
 b_3&=\frac{r(xy-l)+a_2(xt+(c-x)(r+cy)+x^2y-x)+cy(xy-l)}{x(xy-l)}-y.
\end{aligned}
\end{equation}

	Recall that  $xy \neq 0$ and $xy>l$. That is, $xy\geq 1$, and  if  $xy =1$, then  $l = -1$. Further,  $x$ and $y$ always have the same sign.

After collecting the data, we proceed with   the following simple but useful lemma. 
	\begin{lemma}\label{lem:integers}
		Let $a,b \in \Z$, such that $\vert a  \cdot b\vert < a+b$, then $a\in \{0,1\}$ or $b\in\{0,1\}$. Moreover, $a, b\geq 0$.
	\end{lemma} 
	\begin{proof}
		Without loss of generality, we assume that $a\le b$. By assumption $b> 0$. If $a<0$, then $-ab<a+b<b$, which is not possible. Assume $a>  0$. Then 
		\begin{align*}
			ab<a+b\le2b \,&\Rightarrow \, (a-1)b<b \,\Leftrightarrow\, a-1<1
			\,\Rightarrow \, 0< a<2,
		\end{align*}
		hence the result. 
	\end{proof}
 From the lemma we have the following corollary regarding the parameters: 
 \begin{corollary}\label{cor:integers}
     \begin{enumerate}
         \item  Assume that  $de>0$, $cx>0$.  Then $c$ and $e$ have the same sign and $\vert c\vert = 1$, or $\vert e \vert =1$. Furthermore, $t=1, r = 0$ and $\vert y \vert =1$.
         \item If $de<0$ then  $cy+r\ge 0$.
     \end{enumerate}
 \end{corollary}
 \begin{proof}
Let $s = \mathrm{sgn}\,d$. Then 
         \begin{align*}
             0<sd(xy-l)  &= tsx-se(cy+r) \\
             &\Rightarrow  se(cy+r) < tsx.
         \end{align*}
         
     \begin{enumerate}
         \item 
         Since $xy >0$, we have $cy>0$ by the assumption that $cx>0 $. Further, $ed>0$ implies that   $se>0$. Therefore,  $\vert (tse)\cdot(tsc)\vert=\vert ec\vert \le  se \cdot cy \le se(cy+r)<tsx=tse+tsc$.  By Lemma~\ref{lem:integers}, we get $\vert c\vert=1$, $\vert e\vert=1$ and $c, e$, consequently $x, d$, have the same sign.  It immediately follows, that $r=0$ and $\vert y\vert=1$.
         Moreover, since $tsx>0$, $st=\sgn~
         x=\sgn~d=s$. Thus $t=1$.  
         \item Suppose $cy+r<0$. Then $c$ and $y$ have opposite sign  and $0<se(cy+r)$.  Furthermore $cx = c^2 +ce<0$ which implies that   $x$ and $e$ have the same sign. Therefore, $t=-1$ and we get  
         \begin{align*}
             -se<(-se).(-(cy+r)) < -sx = -sc-se
         \end{align*}
         Which is a contradiction, since $-se>0$ and $-sc <0$.   
     \end{enumerate}

Now, we identify some of the cases where we have free actions:
\begin{lemma}\label{Lem:Free_Actions_NNew}
In the following cases we have free actions: 
\begin{enumerate}
    \item $( l, cy)=(1, 1)$. 
    \item $(t, r,l, cy)=(1, 2, 1,-1)$. 
\end{enumerate}
The same holds true for $ey$, replacing $cy$.
     \end{lemma}
 \begin{proof}
     The proof is essentially by the fact that $d=\frac{tx+(c-x)(cy+r)}{xy-l}$ is an integer and hence the parameters appearing in the statement of the lemma result in a free action.  
 \end{proof}
 \begin{lemma}\label{Lem:cy=2=ey}
  Under the assumptions in \eqref{Eq:No_opposite_sign}, we have 
  \begin{enumerate}
  \item $( l, cy)\neq (-1, 1)$. 
  \item $(t, r, l,cy)\neq (1, 2, 1, 2)$. 
  \item $(t, r, l,cy)=(-1, 0, 1, 2)$. 
    \item $(t, r, cy)=(1, 2, -2)$. 
    \end{enumerate}
    The same holds true for $ey$, replacing $cy$.
 \end{lemma}
  \begin{proof}
     Similarly to the proof of Lemma~\ref{Lem:Free_Actions_NNew}, we use the fact that $d=\frac{tx+(c-x)(cy+r)}{xy-l}$ is an integer; however, in this case, this leads to a contradiction with \eqref{Eq:No_opposite_sign}.
 \end{proof}
     
	We will now check the positive curvature condition for $E^7_{p,q}$ and $E^7_{q, p}$. 
 We have the following cases to consider:
	\begin{enumerate}
		\item \label{cases:easy} $0,c+d+e \notin [\min\{c,d,e\}, \max\{c,d,e\}]$, which is equivalent to $c,d,e >0$ 
		\item \label{cases:hard} $c,d,e \notin [0,c+d+e]$, which is equivalent to one of the following six cases:
  \begin{multicols}{2}
      \begin{enumerate}
			\item\label{cases:hard:1} $c<0$, $d,e >c+d+e$ 
			\item\label{cases:hard:2} $d<0$, $c,e >c+d+e$
			\item\label{cases:hard:3} $e<0$, $d,c >c+d+e$
			\item\label{cases:hard:4} $c,d<0$, $e >c+d+e$
			\item\label{cases:hard:5} $c,e<0$, $d >c+d+e$
			\item\label{cases:hard:6} $d,e<0$, $c >c+d+e$
		\end{enumerate}
  \end{multicols}
	\end{enumerate}
 First, we exclude the case where $c+d+e=0$:
\begin{lemma}\label{Lem:Exclue_c+d+e=0}
If the conditions in \eqref{Eq:No_opposite_sign} are satisfied, then $c+d+e\neq 0$.  
\end{lemma}
\begin{proof}
    By contradiction, assume that  $d=-x$. Then we have $$y(c^2-xc+x^2)=(l-t)x+r(x-c).$$
    Depending on the values of $t, r, l$, the right hand side is either $0$, or takes one of the following forms: $2lx$, $2(x-c)$, or $-2c$. Since $\gcd (c, x)=\gcd (c, x-c)=1$, it follows that  $c^2-xc+x^2$ must divide $2$. 
    
    Now observe that the equation $c^2-xc+x^2=L$, for $L=\pm 1, \pm 2$  has no solutions if $\vert c\vert\geq 2$. For $\vert c\vert=1$, it  has integer solutions if and only if $L=1$, yielding  $x=0$, or $x=c$, both of which contradict our assumptions.
\end{proof}
 
 \noindent{Case~\eqref{cases:easy}.} We have $x = c+e >0$, since $c,e>0$. Thus $y>0$. Furthermore, $d>0$. By Corollary~\ref{cor:integers} either $c=1$ or $e=1$, and $r = 0$, $t=1$, $y = 1$.  By Lemma~\ref{Lem:cy=2=ey}, we have $l=1$. Then from Lemma~\ref{Lem:Free_Actions_NNew}  it follows that the action is free.    \\ {Case~\eqref{cases:hard}}. Note that by Corollary~\ref{cor:integers}, in Cases~\eqref{cases:hard:1} and~\eqref{cases:hard:6} we get free actions. In the sequel, we analyze other cases.
 \\{Case~\eqref{cases:hard:2}.} Note that 
\begin{align*}
    (xy-l)(d+c) &= xt+(c-x)(cy+r)+xyc-lc
    =x(t-r)+c(cy+r-l). \\
    (xy-l)(d+e) &= xt-e(cy+r)+xye-le
   =xt+e(ey-r-l).
\end{align*}
First assume that $t=1$. Then since $x=e+c>0$, and $d+e<0$, we must have $ey-r-l<0$. Consequently, $0<ey<r+l$. Thus $l=1, r=2$,  and 
$ey=1$, or $ey=2$. By Lemma~\ref{Lem:cy=2=ey},  $ey\neq 2$, and   by Lemma~\ref{Lem:Free_Actions_NNew}, the action is free if $ey=1$.

Now, let $t=-1$. Then $r=0$. Then $(xy-l)(d+c) 
    =-e+c(cy-l-1),$ and $(xy-l)(d+e) 
   =-c+e(ey-l-1).$ We divide the argument into two cases: $e\geq c$, and $e<c$. Assume first that $e< c$. Then from $ 
    c(cy-l-1)<e<c$ we obtain  $c(cy-l-2)<0$. That is $0<cy<l+2$. Hence $l=1$ and $cy=1$, or $cy=2$. The latter does not occur and the former gives a free action, by Lemmas~\ref{Lem:cy=2=ey}, and \ref{Lem:Free_Actions_NNew}. Similarly, if $e\geq c$, we have the same conclusion. \\                    {Case~\eqref{cases:hard:3}}. By assumptions, we have $e<0$, $x = c+e<0$ and $d+e<0$ with $d,c>0$. By Corollary~\ref{cor:integers}, we have $cy+r \ge 0$ and hence $r = 2$, $t=1$. Then $0>cy \ge -2$. That is   $cy = -1$, or $cy=-2$. The latter does not occur by Lemma~\ref{Lem:cy=2=ey}. If $cy=-1$, then since \begin{align*}
    d = {x-e \over xy-l}\in \Z
\end{align*}
and $xy\neq 0$, we get  $l=1$. By Lemma~\ref{Lem:Free_Actions_NNew}, this gives a free action. \\
Case~\eqref{cases:hard:4}. We have $c,d <0$ and $e>0$. Furthermore $d+x> 0$, and therefore $x>0$.  Consequently, $y >0$. The rest proceeds analogously  to 
Case~\eqref{cases:hard:3}.\\
Case~\eqref{cases:hard:5}. In this case $c<0, e<0$, and hence $x<-2, y<0$.  Moreover,  $$\tilde A\coloneqq x(t-l)+ r(c-x)+y(x^2-cx+c^2)=(xy-l)(d+x)>0$$  
   Since    $x^2-cx+c^2>0$, for $t=1$ we have $(x^2-cx+c^2)y+x(1-l)<0$.  Thus $r=2$. However, in this case we get,  $$x(1-l)+c^2y+(x-c)(xy-2)<0, $$ for  $xy\geq 2$. 
    Therefore,   $t=-1, r=0$, and consequently $l=1$. Then we have $$\tilde A=-2x+c^2y-xcy+x^2y=x(xy-cy-2)+c^2y.$$ It follows that  $xy-cy-2<0$. Since  $y(x-c)>0$,  we can only have $ey=y(x-c)=1$. By Lemma~\ref{Lem:Free_Actions_NNew}, this results in a free action. This finishes the proof of Proposition~\ref{Prop:Cohom_2}.
 \end{proof}

\subsection{Proof of Theorem \ref{thm:main:sing}, Part~\ref{thm:main:sing:smoothsphere}}\label{SS:Part_2}
In this section, we prove  Part~\eqref{thm:main:sing:smoothsphere}. Note that under the assumptions, four of the $\ast_\sigma$'s are regular. Furthermore by Part~\eqref{thm:main:sing:same_Parity} the vertices contained in $\Sigma $ have different parities. Now, we explain the strategy. \\

\noindent{\textbf{Strategy.}}  

Let $\EO$ be an Eschenburg orbifold  such that exactly two vertices $\ast_\sigma$ are contained in $\Sigma$ and that they have different parities. We show that  either the  sphere connecting the  two singular vertices is smooth or $\EO$ does not admit positive curvature. To this end, we use Corollary~\ref{C:Sec>0_Geometric} in the following way: 
   For each edge of the triangles $\Delta(P_1,P_2,P_3)$ and $\Delta(Q_1,Q_2,Q_3)$, we  project all $P_i$'s and $Q_i$'s to the line orthogonal to that edge passing through a particular point (separating axis). We show that the projection of at least one vertex of $\Delta_Q$  lies in $J_P$, the projection of $P$ to the separating axis. 
   
More precisely, we compute the following six quantities for an edge spanned by $Z,Y \in \lbrace P_1,P_2,P_3\rbrace $ or $Z,Y \in \lbrace Q_1, Q_2, Q_3\rbrace$, denoted by $L(Y, Z)$:
\begin{align}\label{Eq:Proj}
    \pr(X) \coloneqq \langle (Z-Y)^\perp, X - W\rangle \text{ for } X \in \lbrace Q_1, Q_2, Q_3, P_1, P_2, P_3\rbrace
\end{align}
Here $v^\perp = (-v_2,v_1)$ for any non zero vector $v \in \R^2$ and $W$ is a particular point chosen accordingly to simplify the computations. 
It will not affect  whether the projections are intersecting or not. 
We collect the projections into a table as follows. Figure~\ref{fig:strat} provides an illustration of the strategy if $\EO$ admits positive curvature.
\begin{center}
\begin{tabular}{cc@{\hspace{5mm}}c@{\hspace{5mm}}}
 \multicolumn{3}{c}{$L(Y, Z)$} \\ \mytableextraspace
      \toprule
 $i\backslash X$ & $Q_i$  & $P_i$ \\  
  \mytableextraspace
\midrule
 $1$ &  $\pr(Q_1)$ & $\pr(P_1)$ \\
 [0.15cm] \mytableextraspace
 $2$ &$\pr(Q_2)$ & $\pr(P_2)$  \\ 
 [.15cm] \mytableextraspace
$3$  &$\pr(Q_3)$ & $\pr(P_3)$
\\  
\mytableextraspace
\bottomrule\\
\end{tabular}
\end{center}

    \begin{figure}[h!]
        \centering
                \caption{Strategy}
        \label{fig:strat}
        
\definecolor{ttqqqq}{rgb}{0.2,0,0}
\definecolor{ffqqqq}{rgb}{1,0,0}
\definecolor{qqzzqq}{rgb}{0,0.6,0}
\resizebox{\textwidth}{!}{
\begin{tikzpicture}[line cap=round,line join=round,>=triangle 45,x=1cm,y=1cm]
\clip(-11,-6) rectangle (11,5);
\draw (-9,4.5) node[anchor=north west] {(1) The line $L(Q_1,Q_2)$ is not separating.};
\draw (1,4.5) node[anchor=north west] {(2) The line $L(P_1,P_3)$ is separating.};
\fill[line width=2pt,color=qqzzqq,fill=qqzzqq,fill opacity=0.1] (-6,3) -- (-7,-1) -- (-3,-2) -- cycle;
\fill[line width=2pt,color=qqzzqq,fill=qqzzqq,fill opacity=0.1] (4,3) -- (3,-1) -- (7,-2) -- cycle;
\fill[line width=2pt,fill=black,fill opacity=0.1] (-8,-3) -- (-4,1) -- (-4,3) -- cycle;
\fill[line width=2pt,fill=black,fill opacity=0.1] (6,3) -- (2,-3) -- (6,1) -- cycle;
\draw [line width=2pt,color=qqzzqq] (-6,3)-- (-7,-1);
\draw [line width=2pt,color=qqzzqq] (-7,-1)-- (-3,-2);
\draw [line width=2pt,color=qqzzqq] (-3,-2)-- (-6,3);
\draw [line width=2p   t,color=ffqqqq] (-4,3)-- (-4,1);
\draw [line width=2pt] (-9.98,-4)-- (-1,-4);
\draw [line width=2pt,dash pattern=on 3pt off 6pt,color=qqzzqq] (-7,-1)-- (-6.986666666666667,-4);
\draw [line width=2pt,dash pattern=on 3pt off 6pt,color=qqzzqq] (-3,-2)-- (-2.9955555555555557,-4);
\draw [line width=2pt,dash pattern=on 3pt off 6pt,color=ffqqqq] (-4,1)-- (-4,-4);
\draw [line width=2pt,dash pattern=on 3pt off 6pt,color=qqzzqq] (-6,3)-- (-6,-4);
\draw [line width=3.2pt,color=qqzzqq] (-6.986666666666667,-4)-- (-2.9955555555555557,-4);
\draw [line width=2pt,dash pattern=on 3pt off 6pt,color=ffqqqq] (-8,-3)-- (-8,-4);
\draw [line width=2pt,color=qqzzqq] (4,3)-- (3,-1);
\draw [line width=2pt,color=qqzzqq] (3,-1)-- (7,-2);
\draw [line width=2pt,color=qqzzqq] (7,-2)-- (4,3);
\draw [line width=2pt,color=ffqqqq] (6,3)-- (6,1);
\draw [line width=2pt,color=ttqqqq] (1.8297595217747782,-6.4307902255384235)-- (9.986845889073436,-1.536538405159229);
\draw [line width=2pt,color=qqzzqq] (5.086167325854393,-4.476945543090655)-- (7.586167325854393,-2.976945543090655);
\draw [line width=2pt,dash pattern=on 3pt off 6pt,color=qqzzqq] (3,-1)-- (5.086167325854393,-4.476945543090655);
\draw [line width=2pt,dash pattern=on 3pt off 6pt,color=qqzzqq] (7,-2)-- (7.586167325854393,-2.976945543090655);
\draw [line width=2pt,dash pattern=on 3pt off 6pt,color=ffqqqq] (6,1)-- (8.17440261997204,-2.6240043666200665);
\draw [line width=2pt,dash pattern=on 3pt off 6pt,color=ffqqqq] (6,3)-- (9.05675556114851,-2.094592601914184);
\draw [line width=2pt,dash pattern=on 3pt off 6pt,color=ffqqqq] (2,-3)-- (3.4685202670308635,-5.447533778384773);
\draw [line width=2pt,color=ffqqqq] (8.17440261997204,-2.6240043666200665)-- (9.05675556114851,-2.094592601914184);
\draw [line width=2pt] (-8,-3)-- (-4,1);
\draw [line width=2pt,color=ffqqqq] (-4,1)-- (-4,3);
\draw [line width=2pt] (-4,3)-- (-8,-3);
\draw [line width=2pt] (6,3)-- (2,-3);
\draw [line width=2pt] (2,-3)-- (6,1);
\draw [line width=2pt,color=ffqqqq] (6,1)-- (6,3);
\begin{scriptsize}
\draw [fill=qqzzqq] (-6,3) circle (2.5pt);
\draw[color=qqzzqq] (-5.930781801239991,3.4) node {$P_1$};
\draw [fill=qqzzqq] (-7,-1) circle (2.5pt);
\draw[color=qqzzqq] (-7.464825121830485,-0.8694668441598349) node {$P_2$};
\draw [fill=qqzzqq] (-3,-2) circle (2.5pt);
\draw[color=qqzzqq] (-2.6435461142603636,-1.7679779319342632) node {$P_3$};
\draw [fill=ffqqqq] (-4,3) circle (2.5pt);
\draw[color=ffqqqq] (-3.9803552936320794,3.4) node {$Q_1$};
\draw [fill=ffqqqq] (-4,1) circle (2.5pt);
\draw[color=ffqqqq] (-3.6,1) node {$Q_2$};
\draw [fill=ffqqqq] (-8,-3) circle (2.5pt);
\draw[color=ffqqqq] (-8.5,-3) node {$Q_3$};
\draw [fill=ffqqqq] (-4,-4) circle (2.5pt);
\draw[color=ffqqqq] (-3.9365254844723507,-4.2) node {$\pr(Q_1) = \pr(Q_2)$};
\draw [fill=qqzzqq] (-6.986666666666667,-4) circle (2.5pt);
\draw[color=qqzzqq] (-6.91695250733388,-4.2) node {$\pr(P_2)$};
\draw [fill=qqzzqq] (-2.9955555555555557,-4) circle (2.5pt);
\draw[color=qqzzqq] (-2.4024821638818574,-3.6) node {$\pr(P_3)$};
\draw [fill=qqzzqq] (-6,-4) circle (2.5pt);
\draw[color=qqzzqq] (-5.843122182920535,-4.2) node {$\pr(P_1)$};
\draw[color=qqzzqq] (-5,-3.8) node {$J_P$};
\draw [fill=ffqqqq] (-8,-4) circle (2.5pt);
\draw[color=ffqqqq] (-7.925038118007633,-4.2) node {$\pr(Q_3)$};
\draw [fill=qqzzqq] (4,3) circle (2.5pt);
\draw[color=qqzzqq] (4.018584878018349,3.4) node {$P_1$};
\draw [fill=qqzzqq] (3,-1) circle (2.5pt);
\draw[color=qqzzqq] (2.594116080327177,-0.9132966533195631) node {$P_2$};
\draw [fill=qqzzqq] (7,-2) circle (2.5pt);
\draw[color=qqzzqq] (7.327735469577842,-1.724148122774535) node {$P_3$};
\draw [fill=ffqqqq] (6,3) circle (2.5pt);
\draw[color=ffqqqq] (5.990926290206127,3.4) node {$Q_1$};
\draw [fill=ffqqqq] (6,1) circle (2.5pt);
\draw[color=ffqqqq] (6.4,1) node {$Q_2$};
\draw [fill=ffqqqq] (2,-3) circle (2.5pt);
\draw[color=ffqqqq] (1.5,-3) node {$Q_3$};
\draw [fill=qqzzqq] (7.586167325854393,-2.976945543090655) circle (2.5pt);
\draw[color=qqzzqq] (8.796034076428743,-3.0828722067261096) node {$\pr(P_1)=\pr(P_3)$};
\draw[color=qqzzqq] (6.5,-4) node {$J_P$};
\draw [fill=ffqqqq] (3.4685202670308635,-5.447533778384773) circle (2pt);
\draw[color=ffqqqq] (4.237733923816991,-5.383937187611841) node {$\pr(Q_3)$};
\draw [fill=qqzzqq] (5.086167325854393,-4.476945543090655) circle (2pt);
\draw[color=qqzzqq] (5.7279474352477555,-4.4854260998374125) node {$\pr(P_2)$};
\draw [fill=ffqqqq] (8.17440261997204,-2.6240043666200665) circle (2pt);
\draw[color=ffqqqq] (8.8836936947482,-2.6) node {$\pr(Q_2)$};
\draw [fill=ffqqqq] (9.05675556114851,-2.094592601914184) circle (2pt);
\draw[color=ffqqqq] (9.738374973362902,-2.1) node {$\pr(Q_1)$};
\end{scriptsize}
\end{tikzpicture}
}
    \end{figure}
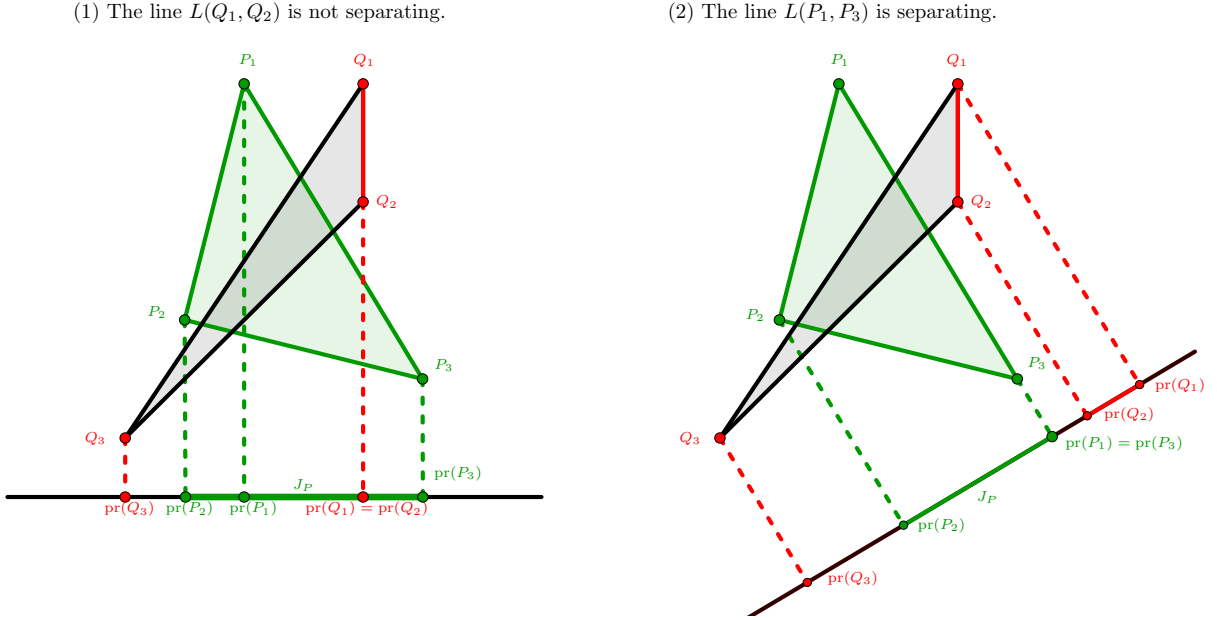

Now, we proceed with the proof.  First note that by a similar argument as in the proof of Part~\eqref{thm:main:sing:same_Parity},  we can assume,  without loss of generality,  that $N_{(123)} = N_{(132)} = N_{(12)} = N_{(13)} = 1$. By Remark~\ref{R:Change_of_Parameteres}, we assume that the parameters are given by 
\begin{align*}
		p = (c,d ,e), \quad q=(c+d+e, 0,0) \quad a = (a_1, a_2, a_3), \quad b=(0,b_2,b_3).
	\end{align*}
Note that the operations led to this representation do not change the isotropy groups.
If $b_2=b_3$, then it follows from  Lemma~\ref{L:b_2=b_3} that the sphere connecting  the two singular points  is smooth. Therefore, as in Part~\eqref{thm:main:sing:same_Parity}, we assume, throughout the proof of this part, that $b_2\neq b_3$. \\

We divide the argument into two cases: 
\begin{enumerate}
    \item $l_{(132)} - l_{(12)} = 0$,
    \item  $l_{(132)} - l_{(12)} \neq 0$. 
\end{enumerate}

We remark that we  follow  the conventions regarding the sign of $l_\sigma$'s
  that we made in the proof of Part~\eqref{thm:main:sing:same_Parity}; however, to avoid repetition, we do not mention them here.\\

\noindent\textbf{Case (1)}. Let $l_{(132)} = l_{(12)} = 1$. Then we have the following parameters:
\begin{align*} 
    p = (1,d ,-1), \quad q=(d, 0,0), \quad a = (a_1, 1, a_3), \quad b=(0,b_2,b_3).
\end{align*}
Now, there are two subcases to consider: $l_{(123)} = l_{(13)}$ and $l_{(123)} = -l_{(13)} $.

If $l_{(123)} = l_{(13)}$,  we get from Equation \eqref{eq:123-13} that $d= -1$ and from Equation \eqref{eq:123} that  
\begin{align*}
    l \coloneqq l_{(123)} &= 1 + a_1 -b_2 -b_3 = -a_3 \\
    \Leftrightarrow a_1 &= l_{(123)} -1 +b_2+b_3
\end{align*}
As a result and  from Equations~\eqref{eq:Id}, \eqref{eq:(23)}, we, respectively, have
\begin{align*}
    l_{Id} &= 1+ l +b_3 -b_2,\\
    l_{(23)} &= 1+ l +b_2 -b_3.
\end{align*}
Note that $N_{11} = \gcd(l+1,b_2-b_3)$.
Let us first collect the torus parameters in this case:

\begin{align} \label{eq:case2:reduction_case_2}
    p = (1,-1 ,-1), \quad q=(-1, 0,0), \quad a = (l-1+b_2+b_3, 1, -l), \quad b=(0,b_2,b_3).
\end{align}

 First observe that if $l=-1$,  we have $l_{Id}=b_3 -b_2=-l_{(23)}$. Then either the action is free, or by a variant of Lemma~\ref{L:b_2=b_3}, the connecting sphere is a  smooth sphere. Therefore, we assume that $l\neq -1$. Then in this case we can directly argue by Lemma~\ref{P:Same_Centroid} and Corollary~\ref{C:Sec>0_Geometric} to show that the Eschenburg orbifolds given by the parameters in \eqref{eq:case2:reduction_case_2}, for $l=1$, do not admit positive curvature. In fact,  a glance at the parameters  reveals that the vertex  $ \begin{pmatrix}
    -1\\0
\end{pmatrix}$ of the triangle $\Delta_Q$ is the midpoint of the edge  joining $\begin{pmatrix}
    -1\\1
\end{pmatrix}$  and $\begin{pmatrix}
    -1\\-1
\end{pmatrix}$ of the triangle $\Delta_P$. Hence by Lemma~\ref{P:Same_Centroid}, the mid point of $L(Q_2, Q_3)$ intersects $\Delta_P$. It is then immediate from Corollary~\ref{C:Sec>0_Geometric} that this  orbifold  cannot admit positive curvature. \\

\noindent
Now let  $l \coloneqq l_{(123)} = -l_{(13)}$.  By Equation \eqref{eq:123-13}, we have $(c+d)(b_2-b_3) = 2l$. Let $c+d = x$ and $b_2 -b_3 =y$. Then $xy = 2l$.  From Equation \eqref{eq:123} we have 
\begin{align*}
    (x-1)a_1 = b_2(x-2)+y+1-l. 
\end{align*}
From Equations~\eqref{eq:Id}, \eqref{eq:(23)},  we get 
\begin{align*}
    l_{Id} &= 1+l -x-y\\
    l_{(23)} & = 1-l+y-x.
\end{align*}
Note that if $x=1$, then $l_{Id}=-l=-l_{(23)}$, hence the action is free. Therefore, we assume that $x\neq 1$ and we get $$a_1=\frac{b_2(x-2)+y+1-l}{x-1}.$$
Here are the parameters of the torus:
\begin{align*} 
    p = (1,x-1 ,-1), \quad q=(x-1, 0,0), \quad a = \left(\frac{b_2(x-2)+y+1-l}{x-1}, 1, \frac{(b_2-1)x-l}{x-1}\right), \quad b=(0,b_2,b_2-y).
\end{align*}

We note that $N_{11} = \gcd(x,y) = 1$. If $x=2$ and $y=l$, then we get $l_{Id}=-1=l_{(23)}$ and the action is free. Consequently, we assume that $x\neq 2$.\\ 

Now we compute the projections:

\begin{center}
\begin{tabular}{ c | c@{\hspace{5mm}}  c|c@{\hspace{6mm}}  c | c@{\hspace{5mm}}  c }
 \multicolumn{3}{c}{$L(Q_1, Q_2)$}& 
      \multicolumn{2}{c}{$L(Q_2, Q_3)$}&\multicolumn{2}{c}{$L(Q_3,Q_1)$}\\ \mytableextraspace

      \toprule

 $i\backslash X$ & $Q_i$  & $P_i$  & $Q_i$  & $P_i$& $Q_i$  & $P_i$\\  \addlinespace[0.2cm] 
\midrule
 $1$ &  $0$ & $l-y-1$ & $0$ & $2(y-l)$& $0$ & $l-y+1$\\ \addlinespace[0.2cm]

$2$ &  $0$ & $1-x$  &  $y-2l$ & $0$ & $2l-y$ & $x-1$\\ \addlinespace[0.2cm]

$3$ &  $2l-y$ & $l + x$ &  $y-2l$ & $-2l$& $0$ & $l-x$\\ \addlinespace[0.2cm]

\bottomrule
\end{tabular}

\vspace{.4cm}

\begin{tabular}{ c | c@{\hspace{5mm}}  c|c@{\hspace{6mm}}  c | c@{\hspace{5mm}}  c }
 \multicolumn{3}{c}{$L(P_1, P_2)$}& 
      \multicolumn{2}{c}{$L(P_2, P_3)$}&\multicolumn{2}{c}{$L(P_3,P_1)$}\\ \mytableextraspace

      \toprule

 $i\backslash X$ & $Q_i$  & $P_i$  & $Q_i$  & $P_i$& $Q_i$  & $P_i$\\  \addlinespace[0.2cm] 
\midrule
 $1$ &  $0$ & $x-2$ & $0$ & $2-2x$& $0$ & $x$\\ \addlinespace[0.2cm]
 
$2$ &  $l+x-y-2$ & $x-2$  &  $1-l-2x$ & $-x$ & $x+y+1$ & $2$\\ \addlinespace[0.2cm]

$3$ &  $y-l+x-2$ & $0$ &  $1+l-2x$ & $-x$& $x-y+1$ & $x$\\ \addlinespace[0.2cm]

\bottomrule
\end{tabular}
\end{center}

Since $xy=2l$ and $x\neq 1, 2$, we have $x=-1$ or $x=-2$. In cases related to  $L(Q_2, Q_3)$ and $L(P_1, P_2)$, we obviously have $0\in J_P$. Since $l=\pm 1$, in the remaining cases, except $L(P_2, P_3)$, simple computations show  that $0\in J_P$. For $L(P_2, P_3)$, we can easily see that both $1-l-2x$ and $1+l-2x$ are contained in $J_P$.\\

\noindent\textbf{Case (2)} Let $l_{(132)} = -l_{(12)}$. If $l_{(123)}= l_{(13)}$, applying $\sigma = (23)$ to $p,q,a, b$ results in the previous case. Therefore, $l \coloneqq l_{(123)} = -l_{(13)}$. We set $x = c+e$ and $y = b_2 -b_3$. By Equation~\eqref{eq:132-12}, we get $xy = 2$. It follows from this and  Equation~\eqref{eq:123-13} that $c+d = xl$. By Equations~\eqref{eq:132} and \eqref{eq:123},  respectively, we have 
\begin{align}
    -1 &= ca_2+x(a_1-b_2)\label{Eq:0 = c(a_2-la_3)_1}\\
    l&=c(a_2-2b_2+y) +b_2 l x +a_1(c-lx)\nonumber \\
    & = c(a_1+a_2-2b_2+y) +lx(b_2-a_1)\nonumber \\
    & = -ca_3+lx(b_2-a_1) \label{Eq:0 = c(a_2-la_3)_2}.
\end{align}
From Equations~\eqref{Eq:0 = c(a_2-la_3)_1} and \eqref{Eq:0 = c(a_2-la_3)_2}  we obtain  $ 0 = c(a_2-la_3)$. Consequently, we have two cases: $c= 0$ and  $a_2 = la_3$. In the latter case, if $l = 1$, then $d =e$ and $a_2 = a_3$. Therefore, we can apply Lemma \ref{L:b_2=b_3} to conclude that  the sphere is smooth.  Whence,  we may assume  $l = -1$. Moreover, from Equation~\eqref{Eq:0 = c(a_2-la_3)_1} and the fact that  $a_1=a_2+a_3+a_1=b_2+b_3=2b_2-y$,  we get $b_2=\frac{-ca_2+1}{x}$.   \\

The parameters of the torus then read as follows:  
\begin{align} \label{eq:case2:reduction_case_3}
    p = (c,-x-c ,x-c), \quad q=(-c, 0,0), \quad a = \left(\frac{-2ca_2}{x}, a_2, -a_2\right), \quad b=\left(0,\frac{-ca_2+1}{x},\frac{-ca_2-1}{x}\right).
\end{align}
Furthermore, we have
\begin{align}\label{actions:nonpos:1}
    l_{Id}=-2{c\over x}, \quad l_{(23)}=2{c\over x}, \quad
    N_{11} = \gcd\left(2,2{c\over x}\right).
\end{align}

Similar as before, in this case, we use Lemma~\ref{P:Same_Centroid} and Corollary~\ref{C:Sec>0_Geometric} to conclude that the Eschenburg orbifold given by the parameters as in \eqref{eq:case2:reduction_case_3} do not admit positive curvature. More precisely,  the vertex  $ \begin{pmatrix}
    -c\\0
\end{pmatrix}$ of the triangle $\Delta_Q$ is the midpoint of the edge  joining $\begin{pmatrix}
    -c-x\\a_2
\end{pmatrix}$  and $\begin{pmatrix}
    -c+x\\-a_2
\end{pmatrix}$ of the triangle $\Delta_P$. Therefore, such Eschenburg orbifolds do not admit positive curvature. 

Now let us assume that  $c = 0$. Then from Equation~\eqref{Eq:0 = c(a_2-la_3)_1}, we have $\vert x \vert = 1$ and $b_2 =a_1 +x$. Further, $y=2x$.  Consequently, we obtain the following parameters for the torus:
\begin{align*} 
    p = (0,lx ,x), \quad q=(x(l+1), 0,0), \quad a = (a_1, a_2, a_1-a_2), \quad b=(0,a_1+x,a_1-x).
\end{align*}
Furthermore, we have
\begin{align}\label{actions:nonpos:2}
    l_{Id} = x(a_1-(1+l)a_2)+1+l, \quad l_{(23)} = x(a_1-(l+1)a_2)-(1+l), \quad N_{11} = \gcd(2,a_1-(l+1)a_2).
\end{align}
If $l=1$ and  $a_1=2a_2$, then by Lemma~\ref{L:b_2=b_3},   the connecting sphere is  smooth.  If $l=-1$ and  $a_1x=\pm 1$, then a simple computations shows that we have  a free action.  Therefore, we assume  $(l, a_1)\neq (1, 2a_2)$ and  $(l, xa_1)\neq (-1, \pm 1)$.\\

Now, we compute the projections:

\begin{center}
\begin{tabular}{ c | c@{\hspace{5mm}}  c|c@{\hspace{6mm}}  c | c@{\hspace{5mm}}  c }
 \multicolumn{3}{c}{$L(Q_1, Q_2)$}& 
      \multicolumn{2}{c}{$L(Q_2, Q_3)$}&\multicolumn{2}{c}{$L(Q_3,Q_1)$}\\ \mytableextraspace

      \toprule

 $i\backslash X$ & $Q_i$  & $P_i$  & $Q_i$  & $P_i$& $Q_i$  & $P_i$\\  \addlinespace[0.2cm] 
\midrule
 $1$ &  $0$ & $l + 1$ & $0$ & $-2  \left(l+1\right)$& $0$ & $l+1$\\ \addlinespace[0.2cm]
 
$2$ &  $0$ & $-xl a_{2}+1+x(a_{1}-a_{2})$  &  $-2  \left(l+1\right)$ & $-2 $ & $2  \left(l+1\right)$ & $xl a_{2}+1+x(-a_{1}+a_{2})$\\ \addlinespace[0.2cm]

$3$ &  $2\left(l+1\right)$ & $xla_{2}+l+x(-a_{1}+a_{2}) $ &  $-2 \left(l+1\right)$ & $-2 l$& $0$ & $-(l+1)x a_{2}+l +a_{1}x$\\ \addlinespace[0.2cm]
\bottomrule
\end{tabular}

\vspace{.4cm}

\begin{tabular}{ c | c@{\hspace{5mm}}  c}
 \multicolumn{3}{c}{$L(P_1, P_2)$}\\ \mytableextraspace

      \toprule

 $i\backslash X$ & $Q_i$  & $P_i$ \\  \addlinespace[0.2cm] 
\midrule
 $1$ &  $0$ & $\left(l a_{2}-a_{1}+a_{2}\right) x$ \\ \addlinespace[0.2cm]

$2$ &  $(1+a_{2}x) l+x(-a_{1}+a_{2})$ & $\left(l a_{2}-a_{1}+a_{2}\right) x$  \\ \addlinespace[0.2cm]

$3$ &  $-x \left(l+1\right) \left(a_{1}-a_{2}\right)+l  \left(xa_{1}-1\right)$ & $0$ \\ \addlinespace[0.2cm]
\bottomrule
\end{tabular}
\vspace{.4cm}

\begin{tabular}{ c | c@{\hspace{5mm}} c@{\hspace{5mm}}|@{\hspace{5mm}} c@{\hspace{5mm}}  c  }
 \multicolumn{3}{c}{$L(P_2, P_3)$}& 
      \multicolumn{2}{c}{$L(P_1, P_3)$}\\ \mytableextraspace

      \toprule

 $i\backslash X$ & $Q_i$  & $P_i$  & $Q_i$  & $P_i$\\  \addlinespace[0.2cm]
\midrule
$1$ & $0$ & $-2x \left(l a_{2}-a_{1}+a_{2}\right) $& $0$ & $x\left(l a_{2}-a_{1}+a_{2}\right)$\\ \addlinespace[0.2cm]
 
$2$ &  $- \left(l-1\right)-2x(\left( l+1\right) a_{2}- a_{1}) $ & $-x\left(l a_{2}-a_{1}+a_{2}\right) $ & $ xl a_{2}-1+x(-a_{1}+a_{2})$ & $0$\\ \addlinespace[0.2cm]

$3$ &   $\left(l-1\right)+2x(\left(- l-1\right) a_{2}+ a_{1})$ & $-x\left(l a_{2}-a_{1}+a_{2}\right) $& $xl a_{2}+1+x(-a_{1}+a_{2})$ & $x\left(l a_{2}-a_{1}+a_{2}\right) $ \\ \addlinespace[0.2cm]
\bottomrule
\end{tabular}
\end{center}
\vspace{.4cm}
To proceed, we divide the argument into two cases: $l=-1$, $l=1$. First, let $l=-1$. Then in all cases except $L(P_2, P_3)$, it is obvious that $0\in J_P$. For  $L(P_2, P_3)$, if $a_1=0$, then $0\in J_p$. Let us assume that $\vert xa_1\vert \geq 2$. If $xa_1\geq 2$, we get $xa_1\leq -2+2xa_1\leq 2xa_1$. If $xa_1\leq -2$, then $2xa_1 \leq 2+2xa_1\leq xa_1$. 

Now let $l=1$.  In all cases, but $L(Q_1, Q_2)$, $L(Q_3, Q_1)$, we can easily see that the projection of at least one vertex of $\Delta_Q$ lies in the projection of $\Delta_P$. For $L(Q_1, Q_2)$, $L(Q_3, Q_1)$, recall that $a_1\neq 2a_2$. Therefore, $(1+x(a_1-2a_2)).(1-x(a_1-2a_2))=1-(a_1-2a_2)^2\leq 0$, i.e. $0\in J_P$. This concludes  the proof of Case (2) and ultimately the proof of Theorem~\ref{thm:main:sing}. 
\end{proof}
The following corollaries follow directly from the proof of Theorem~\ref{thm:main:sing}, Part~\ref{thm:main:sing:smoothsphere}.
\begin{corollary}
    \label{cor:degeneracy} Let $\EO$ have positive curvature  such that four of $\ast_\sigma$ are regular. Then either $P_i = P_j$ or $Q_i = Q_j$ for some $i\neq j$, i.e. there exists a cohomogeneity two action by $S^1\cdot \SU(2)$ on $\EO$. 
\end{corollary}

\begin{corollary}\label{cor:nonnegact}
     Assume that the singular set of $\EO$ is contained in $\bb S^2_{ij}$ and $\EO$ does not have positive sectional curvature. Then $\EO$ is equivalent to one of the following cases:
     \begin{enumerate}
         \item For $b_2, b_3\in \Z$: 
         $$p = (1,-1,-1),  \,\, q = (-1,0,0),   \,\,a = (b_2 + b_3, 1,-1),  \, \, b = (0,b_2,b_3),$$
         $$l_{\Id} = 2+b_3-b_2, \, \, l_{(23)} = 2 +b_2 -b_3,   \,\, N_{11} = \gcd(2, b_2 - b_3).$$
         \item For $(x,y) = (-1, -2l), (-2,-l)$, $l = \pm 1$, $b_2 \in \Z$:
         $$p = (1, x-1,-1), \,\, q  = (x-1,0,0),\,\, a = \left( \frac{b_2(x-2)+y+1-l}{x-1} , 1, \frac{b_2(x-2)-l}{x-1}  \right), \,\, b = (0,b_2,b_2-y), $$
         $$l_{\Id} = 1+l-x-y, \,\, l_{(23)} = 1+l +y-x, \,\, N_{11} = 1.$$
         \item \label{cor:nonnegact:3} For  $0 \neq c, a_2 \in \Z$ and $x = \pm 1, \pm 2$:
        $$p = (c,-x-c,x-c),\,\,q = (-c, 0,0),\,\,a = \left(-\frac{2ca_2}{x}, a_2,-a_2\right),\,\,b=\left(0, \frac{-ca_2
         +1}{x}, \frac{-ca_2-1}{x}\right)$$
    $$  l_{Id}=-2{c\over x},\,\,l_{(23)}=2{c\over x},\,\,N_{11} = \gcd\left(2,2{c\over x}\right)$$
         \item For  $a_1,a_2 \in \Z$, $x = \pm 1$, $l = \pm 1$, $(l,a_1) \neq (1,2a_2)$, $(l, xa_1) \neq (-1, \pm 1)$: $$ p = (0,lx, x), \,q = (x(l+1),0,0),\,\,a = (a_1, a_2, a_1-a_2),\,\,b = (0,a_1 +x, a_1-x)$$
         $$l_{Id} = x(a_1-(1+l)a_2)+1+l,\,\,l_{(23)} = x(a_1-(l+1)a_2)-(1+l),\,\,N_{11} = \gcd(2,a_1-(l+1)a_2)$$
         
     \end{enumerate}
     Furthermore, $\EO$ has precisely one singular point with local group $\Z_k$ in the following cases:
    \begin{enumerate}
        \item[(1)] $b_2 - b_3 = \pm 1, \pm 3$, $ k = 3, 5$.
        \item[(2)] $l = -1$ and $(x,y) = (1,2), (-2,1)$, $k = 3$.
        \item[(4)] $l = 1$ and $a_1 = 2a_2\pm x$, with $x = \pm 1, \pm 3$ and $k = 3,5$.
    \end{enumerate}
\end{corollary}
\begin{remark}
    Note that in case~\ref{cor:nonnegact:3}, the local groups are not cyclic in general, for example if $x = \pm1$ and $c,a$ are odd. In all other cases the local groups are cyclic.
\end{remark}

\begin{corollary}
\label{lem:actionssmooth} 
    Let $\EO$ be an effective Eschenburg orbifold such that the singular set is a smooth orbifold sphere $\bb S_{ij}^2$ with cyclic isotropy group $\bb Z_k$. Then the action is equivalent to one of the following cases
    
    \begin{enumerate}
        \item \label{lem:actionssmooth:1} For $s, u \in \bb Z$:
        $$p = (0,1,-1), \,\,q = 0, \,\, a = (u-1, s, u+1-s), \,\, b = (0,u,u),$$
         and $k=1-u$.
        \item  \label{lem:actionssmooth:2} For $m,r,c,e \in \bb Z$ such that $c | em+1$: 
\begin{align*}
    p &= (c, cr-e,e),\, \,q = (c(r+1),0,0),\\ a &= \left(mr, m + {em +1\over c}, m(r+1) - {em+1\over c}\right),\, \, b = (0, m(r+1), m(r+1)),
\end{align*}
and $k = r$.
        \item  \label{lem:actionssmooth:5} For $s, u \in \bb Z$:
$$p = (0,1,1), \,\,q = (2,0,0), \,\,a = (u+1,s, u-1-s), \,\, b = (0,u,u),$$
and  $k = u-2s-1$.
        \item  \label{lem:actionssmooth:6} For $r, c,e,m\in \Z$ such that  $c|1-em$:  \begin{align*}
            p &= (c, e-cr,e),\,\,q = (2e+c(1-r), 0,0),\\
        a &= \left(2{1-em\over c}+mr,{1-em \over c} -m, {1-em \over c}+m(r-1)\right),\\ b &=\left(0, 2{1 -em \over c} + m(r-1), 2{1-ek \over c} + m(r-1)\right), 
        \end{align*}
        and, $k= r$. 
        \item  \label{lem:actionssmooth:nonpos:1} For $a_2 \in \Z$: $$p = (1,-2,0),\,\,q = (-1,0,0),\,\,a = (-2a_2,a_2,-a_2),\,\,b=(0,-a_2+1,-a_2-1),$$
        and  $k= 2$. 
    \item \label{lem:actionssmooth:nonpos:2} For $a_2 \in \Z$, $\vert x \vert = 1$ and $\vert a_1\vert =2$:
    $$p= (0,-x,x),\,\,q = 0,\,\,a = (a_1,a_2,a_1-a_2),\,\,b=(0,a_1+x,a_1-x),$$
    and  $k = 2$.  
    \end{enumerate}    
    Furthermore, in cases \eqref{lem:actionssmooth:1} -- \eqref{lem:actionssmooth:6} $\EO$ admits positive sectional curvature in cases~\eqref{lem:actionssmooth:nonpos:1} and~\eqref{lem:actionssmooth:nonpos:2} it does not.
\end{corollary} 
\begin{remark}
    Note that $c \vert \, me+1$ implies, that $m$ and $c$ are relatively prime.
\end{remark}
\begin{proof}
    If $\EO$ has positive sectional curvature, then from Corollary~\ref{cor:degeneracy} we either have $q_2=q_3$, $b_2=b_3$, or $p_i=p_j$, $a_i=a_j$, $i\neq j$. Therefore, up to an equivalence, we can assume that 
       \begin{align*}
		p = (c,d ,e), \quad q=(c+d+e, 0,0) \quad a = (a_1, a_2, a_3), \quad b=(0,u,u).
	\end{align*}
 Moreover, the orders of the isotropy groups at $T_\sigma$ are given by 
 \begin{align}
    l_{Id} &= -(d+e)(a_2-u)-da_1 = l_{(23)}. \label{eq:smooth:1}\\
         l_{(132)} &= ca_2 + (c+e)(a_1 - u)= l_{(12)}.  \label{eq:smooth:2}\\
		l_{(123)} &= c(a_2-u) - d (a_1 - u)=l_{(13)}. \label{eq:smooth:3}
    \end{align}
    We have the following equations:
    \begin{align}
        l_{(132)} - l_{(123)} &= -(d+e)u+(c+d+e)a_1. \label{eq:smooth:-}\\
        l_{(132)} + l_{(123)} & = c(2a_2-u)+ (c+e-d)(a_1-u). \label{eq:smooth:+}
    \end{align}
    We first assume $l_{(132)}=l_{(123)}=1$: Let $c = 0$. If $d+e \neq 0$, we get from equation \eqref{eq:smooth:-}, the $a_1 = u$ and therefore, $l_{(132)} =0$, a contradiction. Thus we assume $e = -d$. From equation \eqref{eq:smooth:3}, we get $d \in \pm 1$ and $a_1 -u  = -d$. By multiplying all parameters by $d$, we can assume $d = 1$. Let $s = a_2$, then we have the parameters of case \eqref{lem:actionssmooth:1}.\\ Now we assume $c \neq 0$.
    By equation \eqref{eq:smooth:-}, we have that $c |(d+e)(a_1-u)$. By equation \eqref{eq:smooth:2} $c$ and $a_1-u$ are relatively prime. Therefore $c | d+e $ and $d = cr-e$ for some $r \in \bb Z$. From equation \eqref{eq:smooth:-}, we have 
    \begin{align*}
       0 = -cru + c(r+1)a_1
    \end{align*}
    Therefore $r|a_1$. Let $m \in \bb Z$, such that $a_1 = mr$, then $u = m(r+1)$. By equation \eqref{eq:smooth:2}, we get 
    \begin{align*}
        ca_2 = 1+(c+e)m
    \end{align*}
    Therefore $c|em+1$ and $a_2 = m+{em+1 \over c}$. We get the parameters of case \eqref{lem:actionssmooth:2}.\\
    Now Suppose $l_{(132)} = -l_{(123)}=1$. If $c = 0$ and $e\neq d$, we get $a_1 = u$ and therefore $l_{(132)}=0$, a contradiction. Hence we assume $e = d$. Again by equation \eqref{eq:smooth:3}, we have $d = \pm 1$ and $a_1 - u = d$. By multiplying all parameters with $d$, we can assume $d = 1$. Therefore $a_1 = u+1$. With $s = a_2$, we get the parameters of case \eqref{lem:actionssmooth:5}. If $c \neq 0$, then by equation \eqref{eq:smooth:+}, $c|(e-d)(a_1-u)$. Since by equation \eqref{eq:smooth:2}, $a_1-u$ and $c$ are relatively prime, we have $c\vert e-d$ and therefore there exists $r \in \bb Z$, such that $d = e-cr$. Let $m = a_1 -u$, then by equation \eqref{eq:smooth:2}, we have
    \begin{align*}
       1 = ca_2+(c+e)m
    \end{align*}
    Therefore $c| (1-em)$ and $a_2 = {1-em \over c} -m$. Furthermore by equation \eqref{eq:smooth:-}, we get $a_1 = mr +2{1-em \over c}$ and $u = 2{1-em \over c}+m(r-1)$, which are the parameters in case \eqref{lem:actionssmooth:6}. 
    \\
    
   Now assume that $\EO$ does not admit positive curvature. Then for $i\neq j$, we have  $P_i\neq P_j$ and $Q_i\neq Q_j$. In other words, none of the triangles are degenerate. 
It follows from the proof of Part~\eqref{thm:main:sing:smoothsphere}, particularly from equations~\eqref{actions:nonpos:1} and~\eqref{actions:nonpos:2}, that The only actions in which the remaining sphere is smooth are given by  the parameters in cases~\eqref{lem:actionssmooth:nonpos:1} and~\eqref{lem:actionssmooth:nonpos:2}.
\end{proof}

\section{Cohomology of $6$-dimensional Eschenburg Orbifolds}\label{sec:cohom}

This section is dedicated to the computation of the orbifold cohomology of an Eschenburg $6$-orbifold $\EO$. First, we recall the definition of orbifold cohomology and explain how to compute the cohomology of a biquotient using Eschenburg's method \cite{Esch92}. Following this approach, we then compute the cohomology ring  of $\EO$ and prove Theorem~\ref{thm:main:cohom}. Finally, we   compute the cohomology groups of  Eschenburg orbifolds appearing in Lemma~\ref{lem:actionssmooth} and prove Theorem~\ref{thm:main:cohomgroups}.

\subsection{Preliminaries on Orbifold Cohomology}\label{subsec:orbicohom} 

For a compact Lie group $G$ we denote by $BG$ its classifying space and by $EG$ a contractible space on which $G$ acts freely  such that $EG/G \cong BG$. We have a principal $G$-bundle
\begin{align*}
    G \to EG \to BG.
\end{align*}
If $G$ acts on a smooth manifold $M$, its equivariant cohomology  is given by the cohomology of the Borel construction $EG \times_G M$
\begin{align*}
    H^\ast_G (M) \coloneqq H^\ast(EG \times_G M). 
\end{align*}
For a quotient orbifold $\mathcal{O} = M/G$ we define its orbifold cohomology to be
\begin{align}\label{dfn:orbicohom}
    H^\ast_\mathcal{O}(\mathcal{O}) \coloneqq H^\ast_G(M)
\end{align}
\begin{remark}
    Any effective orbifold is a  quotient orbifold $\mathrm{Fr}\,\mathcal{O}/\O(n)$, where $\mathrm{Fr}\,\mathcal{O}$ denotes the frame bundle of $\mathcal{O}$. By \cite[Corollary 1.52.]{ALR07} the orbifold cohomology does not depend on the representation of $\mathcal{O}$ as a quotient orbifold and hence~\eqref{dfn:orbicohom} can be used as a definition for orbifold cohomology of effective orbifolds. As already mentioned in Remark~\ref{rem:orbifoldquotient}, the situation for ineffective orbifolds is different and needs the language of groupoids, since it is not known  whether any ineffective orbifold is a quotient. It is still possible to define orbifold cohomology in this case (see \cite[p. 38]{ALR07}), which is equivalent to \eqref{dfn:orbicohom} when restricting to quotient orbifolds \eqref{dfn:orbicohom} (see    \cite[Example~2.11]{ALR07}). Since in this work we only consider quotient orbifolds, we take \eqref{dfn:orbicohom} as a definition. For more details on this topic, we refer the reader to \cite{ALR07}.
\end{remark}

We will now describe Eschenburg's method for computing the cohomology of biquotients: Let $G$ be a compact Lie group. $G \times G$ acts on $G$ by $(g_1,g_2) \cdot g = g_1 g g_2^{-1}$. Let $U \subset G\times G$ be a closed subgroup and denote the quotient of $G$ by the subaction of $U$ as $G\sslash U$. If $U$ acts freely then $G\sslash U$ is a manifold, and if the action of $U$ is almost free then $G\sslash U$ is an orbifold. Such spaces are called \emph{(orbifold) biquotients}. In this sense the effective Eschenburg $6$-orbifolds $\EO$ are orbifold biquotients. It is convenient to allow $U \to G\times G$ to  just be an immersion. Then also the ineffective Eschenburg orbifolds are orbifold biquotients. Actually, Eschenburg's method works for  computing the equivariant cohomology of the $U$-action on $G$  even if $U \to G \times G$ is not even an immersion. \\

Let $\phi \colon U \to  G\times G$ be a smooth homomorphism of Lie groups. We get an induced  map of bundles
\begin{equation} \label{eq:bundles}
\begin{tikzcd}
	U \arrow[r] \arrow[d] & EU \arrow[r] \arrow[d, "E\phi"] & BU \arrow[d, "B\phi"] \\
	G\times G \arrow[r]   & EG\times EG \cong EG^2 \arrow[r]                    & BG\times BG \cong BG^2,                 
\end{tikzcd}
\end{equation}
where the map $U \to G\times G$ is homotopy equivalent to $\phi$ in the following way: Let $\rho \colon H \to L$ be a homomorphism between  Lie groups. Let $H$ act on $EH \times EL$ via the diagonal action
\begin{align*}
    h \ast (v,w) = (hv, \phi(h)w)
\end{align*}
Since $H$ acts freely on $EH$, this action is free and therefore  $EH\times_H EL \simeq BH$. Since $EL$ is contractible, projection to the second factor gives the following bundle map, such that the first vertical map is homotopy equivalent to $\rho$
\begin{equation*} 
\begin{tikzcd}
	H \arrow[r] \arrow[d] & EH \arrow[r] \arrow[d] & BH \arrow[d] \\
	L \arrow[r]   & EL \arrow[r]                    & BL.               
\end{tikzcd}
\end{equation*}
\noindent
In total 
\begin{align} \label{cohom:bundle}
    G \to EU \times_UG \to BU
\end{align}
is the pullback bundle of the reference bundle 
\begin{align}\label{cohom:refbundle}
    G \to EG^2 \times_{G^2}G\to BG^2.
\end{align}
Eschenburg's method to compute the Serre spectral  sequence of (\ref{cohom:bundle}) relies on the Serre spectral sequence of the reference bundle (\ref{cohom:refbundle}). We make the following assumptions on $G$: Let the coefficient ring $R$ be $\bb Z$ or a field\\
\begin{center}
    \begin{minipage}[c]{.5 \textwidth}
$H^\ast(G)$ is a free exterior algebra $\Lambda[a_1, \ldots a_m]$ with generators $a_1, \ldots a_m$ of degrees $\deg a_j = r_j -1$, where $r_j$ is even with $2 \le r_1 \le r_2 \le \ldots \le r_m$. \\
\end{minipage}
\end{center}

By \cite{B53}, we have $H^\ast (BG) = R[\bar a_1, \ldots, \bar a_m]$, with $\deg \bar a_j = r_j$. The generators $\bar a_i$ can be chosen to correspond to the $a_i$ under transgression in the spectral sequence of $EG \to BG$. Furthermore let $T \subset G$ be a maximal torus and $t_1,\ldots, t_k$ a basis of $H^1(T)$. Then $H^\ast(BT) = R [\bar t_1, \ldots, \bar t_k]$. If $W = N_G(T)/T$ is the Weyl group of $G$, then $H^\ast(BG) = H^\ast(BT)^W$ the subalgebra of Weyl group invariant polynomials (\cite{B53}, Prop. 27.1).  Now let $\rho \colon G \to G^\prime$ be a homomorphism of Lie groups with maximal tori $T$ and $T^\prime$ and $\rho_T = \rho\vert _T$. Denote by $\rho_{T\ast}$ its differential, then $\rho_T^\ast \colon H^1(T^\prime) \to H^1(T)$ is given by  the adjoint map of $\rho_{T\ast}$. Hence  $B\rho^\ast \colon H^\ast (BG^\prime) \to H^\ast(BG)$ is given by the extension of $\rho_T^\ast$ to the Weyl-invariant polynomials over $H^1(T^\prime)$ and $H^1(T)$.
\begin{example}
    We have 
    \begin{align*}
        H^\ast(\U(n)) = \Lambda[\alpha_1,\ldots, \alpha_n] \quad \text{ and } \quad H^\ast(B\U(n) = R[\sigma_1^n, \ldots, \sigma_n^n] \subset R[\bar t_1, \ldots, \bar t_n] = H^\ast (BT),
    \end{align*}
    where 
    \begin{align*}
        \sigma_1^n = \sum_j \bar t_j,\quad \sigma_2^n = \sum_{j<k} \bar t_j \bar t_k,\quad \ldots\quad,\quad \sigma_n^n = \bar t_1\dots \bar t_n.
    \end{align*}
\end{example}
For the Serre spectral sequence $E_r$ of the bundle (\ref{cohom:bundle}) let $k_r \colon H^\ast(BU) \to E^{\ast,0}$ be the projection onto the horizontal axis and define $\delta_i \coloneqq \bar a_i \otimes 1 - 1 \otimes \bar a_i \in H^\ast(BG^2) = H^\ast(BG) \otimes H^\ast(BG)$. We have the following theorem.
\begin{theorem}[\cite{Esch92}, Theorem 1] \label{thm:cohom:esch}
    Let $a_1, \ldots, a_j$ be the generators of $H^\ast(G)$. In the spectral sequence $E_r$ of the bundle (\ref{cohom:bundle})  all $1 \otimes a_j \in E_2^{0,\ast}$ are transgressive and 
    \begin{align*}
        d_{r_j}(1 \otimes a_j) = k_{r_j}(B\phi^\ast \delta_j).
    \end{align*}
\end{theorem}
\begin{remark}
    Note that Eschenburg required the action of $U$ on $G$ to be free. Together with (\ref{eq:bundles}) and (\ref{cohom:bundle}), Theorem~\ref{thm:cohom:esch} still holds even if the action is neither free nor  effective. The freeness of the $U$-action on $G$, implies that $U\sslash G$ is a manifold and furthermore
    \begin{align*}
        EU\times_UG \simeq G\sslash U.
    \end{align*}
    Therefore, the singular cohomology ring of $G\sslash  
    U$ is recovered. Since we are interested in the orbifold cohomology as defined in (\ref{dfn:orbicohom}), we do not need this requirement.
\end{remark}

If not stated otherwise, from now on we assume the coefficient ring $R$ of any cohomology theory to be the integers $\bb Z$.

\subsection{Proof of Theorem~\ref{thm:main:cohom}}
    We cannot directly apply Theorem ~\ref{thm:cohom:esch} since $\T$ is not a subset of $\SU(3)^2$. Instead we define
    \begin{align*}
        \phi \colon T^3 \to T^3 \times T^3 \subset \U(3) \times \U(3), \\
        (z,w,u) \mapsto (z^p \cdot w^a, z^q \cdot w^b\cdot u^{(1,0,0)})
    \end{align*}
    where $z^v = \diag(z^{v_1},z^{v_2},z^{v_3})$ for $v \in \bb Z^3$.
    With these definitions
    \begin{align*}
        \EO \cong \U(3)\sslash T^3 \quad \text{ and } \quad E\T \times_{\T} \SU(3) \simeq ET^3 \times_{T^3} \U(3)
    \end{align*}
    Therefore we can apply Theorem \ref{thm:cohom:esch} to $G = \U(3)$ and $\phi \colon T^3 \to \U(3)\times \U(3)$ to compute the Serre spectral sequence of the bundle 
    \begin{align*}
        \U(3) \to ET^3 \times_{T^3} \U(3) \to BT^3
    \end{align*}
   We start with determining $B\phi^\ast$. It is easy to see that $\phi_\ast \colon H^1(T^3) \to H^1(T^3 \times T^3) = H^1(T^6)$ is given by the matrix:
   \begin{align*}
       \phi_\ast = \begin{pmatrix}
           p_1& a_1&0 \\
           p_2&a_2&0\\
           p_3&a_3&0\\
           q_1&b_1&1\\
           q_2&b_2&0\\
           q_3&b_3&0
       \end{pmatrix} \text{ and therefore } \phi^\ast = \begin{pmatrix}
           p_1 & p_2 &p_3 &q_1 &q_2 &q_3\\
           a_1& a_2 &a_3 &b_1&b_2 &b_3\\
           0&0&0&1&0&0
       \end{pmatrix}
   \end{align*}
    By the discussions in the previous subsection $B\phi^\ast \colon H^2(BT^6) \to H^2(BT^3)$ is given by the same matrix as $\phi^\ast$. For this let $s_1,s_2,s_3,t_1, t_2, t_3$ be the generators of $H^\ast (BT^6) = \bb Z[s_1,s_2,s_3,t_1, t_2, t_3]$ and $s,t,r$ the generators of $H^\ast(BT^3) = \bb Z [s,t,r]$, then:
    \begin{align*}
        B\phi^\ast(s_i) &= p_i s + a_i t\\
        B\phi^\ast(t_1) &= q_1 s + b_1 t +r\\
        B\phi^\ast(t_i) &= q_i s + b_i t, \quad i = 2,3
    \end{align*}
    And therefore:
    \begin{align*}
        B\phi^\ast(\delta_i) &= B\phi^\ast(\sigma_i \otimes1-1 \otimes\sigma_i) \\
        B \phi^{\ast}(\delta_{1})&=-r,\\
        B \phi^{\ast}(\delta_{2})&=\sigma_2(sp+ta)-\sigma_2(sp+ta)-A_1r\\
B \phi^{\ast}(\delta_{3})&=\sigma_3(sp+ta)-\sigma_3(sp+ta)-A_2r\\
    \end{align*}
    For two homogeneous polynomials $A_i\in \bb Z\lbrack t,s\rbrack$ of degree $2i$. Now we compute the Serre spectral sequence of the bundle \ref{cohom:bundle}:\\
    Let $a_{1}, a _{3}, a_{5}$ be the generators of $\H^{\ast}(\UU(3), \Z)$ and $k_r \colon H^\ast (B\UU(3)) \to E_r^{\ast,0}$. By \ref{thm:cohom:esch} we know that
\begin{align}
    d_{2i}(1 \otimes a_{2i-1}) &= k_{2i}(\CB \phi^\ast \delta_{2i})\\
    d_2(1\otimes a_1)&= -k_2(r)\label{rel:1}\\
    d_4(1\otimes a_3)&= k_4(\sigma_2(sp+ta)-\sigma_2(sp+ta))-k_4(A_1(t,s)r)\label{rel:2}\\
    d_6(1 \otimes a_5) &= k_6(\sigma_3(sp+ta)-\sigma_3(sp+ta))-k_6(A_2(t,s)r) \label{rel:3}
\end{align}
Since the odd degree cohomology of $BT^3$ vanishes, the differentials in odd degrees also vanish. By relation (\ref{rel:1}), we have $E_4^{\ast,0} \cong \bb Z\lbrack s,t\rbrack$. Since $d_2$ is injective, it has no kernel and therefore $E_4^{0,1} = 0$. By the same reasoning and relation (\ref{rel:2}) we have $E_6^{\ast,0} \cong \bb Z\lbrack s,t\rbrack/\langle \sigma_2(sp+ta)-\sigma_2(sp+ta)\rangle $ and $E_6^{0,3} = 0$. Again by relation (\ref{rel:3}) we have $E_7^{\ast,0} \cong \bb Z\lbrack s,t\rbrack/\langle \sigma_2(sp+ta)-\sigma_2(sp+ta), \sigma_3(sp+ta)-\sigma_3(sp+ta)\rangle $. We will now show, that also $E_7^{0,5} = 0$, which is equivalent to $d_6$ being injective on $E_6^{0,5}$. Suppose $d_6$ has a non trivial kernel, then $E_7^{5,0} \cong \bb Z$. Since all other groups on the diagonal between $E_7^{5,0}$ and $E_7^{0,5}$ vanish and $E_7^{5-p,p} = E_\infty^{5-p,p}$, we must have $H^5 (ET^3\times_{T^3}\UU(3)) \cong \bb Z$. Since $\UU(3) \sslash T^3$ is an orbifold, its cohomology ring with rational  coefficients satisfies Poincaré-duality (cf. \cite[Proposition 1.28]{ALR07}). Therefore $\bb Q \cong H^5 (ET^3\times_{T^3}\UU(3),\Q) \cong H_1(ET^3\times_{T^3}\UU(3),\Q)$. Since the odd degree cohomology of $H^\ast (BT^3)$ vanishes, we have $H^1(ET^3\times_{T^3}\UU(3),\Q) = 0$ and the first rational betti number vanishes, a contradiction. Therefore $d_6$ is also injective and the Serre spectral sequence collapses on page $7$. Furthermore all entries are trivial except for $E_\infty^{\ast,0} = \bb Z \lbrack s,t\rbrack/\langle \sigma_2(sp+ta)-\sigma_2(sp+ta), \sigma_3(sp+ta)-\sigma_3(sp+ta)\rangle $, which is therefore isomorphic to $H^\ast(ET^3\times_{T^3}\UU(3)) = H_{\OO}(\EO)$.\hfill \qed
\begin{remark}
    Although the relations in the cohomology ring seem quite simple, it is in general hard to determine the cohomology groups of the Eschenburg orbifolds for entire subfamilies. In particular, in a given even degree $2m$, the relations define an $(m+1)\times (m+1)$-matrix $A =A_{p,q}^{a,b}$ with $\bb Z$-coefficents. The cohomology group in degree $2m$, is isomorphic to 
    \begin{align*}
        \bb Z^{m+1} / \mathrm{Im}\, A \cong \bigoplus_{i =1}^k\bb Z \oplus \bigoplus_{j = 1}^{l} \bb Z_{r_j}
    \end{align*}
    In order to find this isomorphism, it is necessary to bring $A$ into a diagonal form, its Smith normal form.  Given specific parameters, $p,q,a,b$ it is in general easy for a computer program to execute this task. Unfortunately, it turns out to  be a very hard task, to do this for entire families. We were able to compute the cohomology groups in all degrees in the special cases, where the singular set is an edge of the graph in Figure \eqref{fig:sing:locus}. Since these cases have special behaviour in positive curvature, it is interesting to study their cohomology rings.
\end{remark}
\subsection{Proof of Theorem~\ref{thm:main:cohomgroups}}

\begin{proposition}
\label{prop:isom}
    Let $M$ be a closed   connected manifold of dimension $n$ and $G$ be a compact  Lie group acting on $M$. Let $F \subset M$ be a closed $G$-invariant submanifold of  codimension $k$  containing the set of points with nontrivial isotropy groups. Assume further that the normal bundle of $F$ is orientable and   $H^{n - \dim G -1}((M \backslash F)/G) = 0$. Then there exist $u \in H^k_G(\mathcal{V} F,\mathcal{V} F\backslash F)$, such that 
 \begin{align*}
  H^{i}_G(F) \xrightarrow{\cup u} H^{i+k}_G(\mathcal{V} F,\mathcal{V} F\backslash F) \xrightarrow{i^\ast (j^\ast)^{-1} } H^{i+k}_G(M)
 \end{align*}
 is an isomorphism for $i\ge n-\dim G-k $ and surjective for $i = n - \dim G-k -1$. Here $ i \colon (M ,\emptyset) \to (M, M \backslash F)$ is the inclusion and $j \colon   (\mathcal{V} F, \mathcal{V} F \backslash F) \to (\mathcal{T} F, \mathcal{T} F\backslash F) \to (M, M\backslash F)$. The first arrow is an isomorphism and the second one the inclusion of a tubular neighbourhood $\mathcal T F$ in $M$. 
\end{proposition} 
\begin{proof}
    Since the normal bundle of $F$ is orientable, the same holds for the Borel-construction $EG \times _G \mathcal{V} F$. Therefore there exists $u \in H^k(EG \times_G\mathcal{V} F,EG \times_G (\mathcal{V} F\backslash F)) = H^k_G(\mathcal{V} F,\mathcal{V} F\backslash F)$ such that 
    \begin{align*}
        \cup u \colon H^i_G(F) = H^i(EG \times_G F) \to  H^{i+k}(EG \times_G\mathcal{V} F,EG \times_G (\mathcal{V} F\backslash F)) = H^{i+k}_G(\mathcal{V} F,\mathcal{V} F\backslash F)
    \end{align*}
  is an isomorphism, given by the Thom isomorphism.
  Moreover, 
  by excision we have $H^\ast_G(M,M\backslash F) \cong H_G^\ast(\mathcal{T} F, \mathcal{T} F\backslash F) \cong H_G^\ast(\mathcal{V} F, \mathcal{V} F\backslash F)$. 
  Furthermore,  there is a long exact sequence of equivariant cohomology groups for the pair $(M, M\backslash F)$:
    \begin{align*}
        \ldots \to H_G^{r-1}(M\backslash F) \xrightarrow{\partial} H^r_G(M,M\backslash F) \xrightarrow{i^\ast} H^r_G(M) \xrightarrow{} H^{r}_G(M\backslash F) \to \ldots
    \end{align*}
    Since $G$ acts freely on $M \backslash F$ we have $EG \times_G (M \backslash F) \simeq (M \backslash F)/G$. Furthermore, $ (M \backslash F)/G$ is a  manifold of dimension $n - \dim G$ without compact components. 
    Therefore $H^r_G(M\backslash F) \cong H^r((M \backslash F)/G)$ is trivial for $i \ge n - \dim G$.  
    By assumption we also have $H^{n- \dim G -1}_G(M\backslash F) \cong H^{n-\dim G -1}((M \backslash F)/G) = 0$ and it follows from the long exact sequence, that $i^\ast$ is an isomorphism for $i \ge n- \dim G -k$ and surjective for $i = n -\dim G - k -1$. Now the result follows immediately. 
\end{proof}

\begin{remark}[Cohomology of teardrops.] \label{rem:singpoint}
    Let $\bb S^2 \cong \U(2)\sslash T^2$ be an orbifold $2$-sphere by a biquotient action of $T^2$ on $\U(2)$ with only a single singular point with finite orbifold group $H$. If we pick $M = \U(2)$, $F = \U(2)^H\cong T^2$ and $G = T^2$ then $(M\backslash F)/G = \bb S^2 \backslash\lbrace \ast \rbrace \simeq \lbrace \ast \rbrace$ and therefore $H^{n -\dim G-1}((M\backslash F)/G ) = H^1((M\backslash F)/G ) = 0$. Hence we can apply Proposition \ref{prop:isom} and get
    \begin{align*}
        H_\OO^i(\bb S^2) \begin{cases}
            0 & \text{for } i \text{ odd}\\
            \Z & \text{for } i=0,2\\
            H &\text{for } i \ge 4, \text{ even}
        \end{cases}
    \end{align*}
\end{remark}

\begin{theorem}\label{thm:cohomology:singsphere}
\label{thm:cohomology}
  Let $\EO$ be and effective Eschenburg orbifold and   $\Sigma = \bb S^2_{ij}$  contains the singular set of $\EO$.   Then we have 
  \begin{align*}
                H^i_{\mathcal{O}}(\EO) = \begin{cases}
               0 & \text{for } i \text{ odd}\\
                   \bb Z & \text{for } i = 0\\
                    \bb Z^2 &\text{for } i = 2,4\\
               H^{i-4}_\mathcal{O}(\bb S^2) & \text{for } i \ge 6
          \end{cases}
     \end{align*}   
\end{theorem}
\begin{remark}
    The results of Theorem \ref{thm:cohomology} in degrees larger than $6$ also follow from a localization argument (see \cite{tD87}). However, to get the conclusion for the sixth cohomology  group  we need Proposition \ref{prop:isom}.
\end{remark}
\begin{proof}
    Our goal is  to apply Proposition~\ref{prop:isom} to the triple $(M, F, G)=(\SU(3), \U(2)_{ij}, T^2)$ to prove the theorem as $ H^i_{\mathcal{O}}(\EO)= H^i_{T^2}(\SU(3))$.  To apply Proposition~\ref{prop:isom}, observe that all conditions except for $H^{n - \dim G -1}((M \setminus F)/G) = 0$ are clearly satisfied. It remains to show that this condition is also satisfied, i.e. $H^5((\SU(3)\setminus \U(2)_{ij})/T^2) = 0$. In what follows we first demonstrate that $H^5((\SU(3)\setminus \U(2)_{ij})/T^2) = 0$ and then we conclude the proof of the  theorem. Define
 \begin{align*}
     \rho\colon \SU(3)&\to \SU(3)/\U(2)\\
     A&\mapsto \lbrack\tau_i A\tau_j\rbrack,
 \end{align*}
 where $\tau_i$ is defined right after \eqref{Eq:U(2)_ij}. Note that  $\rho$ is a $\U(2)$-fiber bundle over $ \bb C \bb P^2\cong \SU(3)/\U(2)$ with $\rho(\U(2)_{ij})=\lbrack I \rbrack=\colon \ast $. Therefore, we get the following fibration
 \begin{align*}
            \U(2) \xrightarrow{} \SU(3) \setminus \U(2)_{ij} \xrightarrow{\pi} \bb C \bb P^2 \setminus \lbrace \ast \rbrace. 
        \end{align*}
        
Since $\bb C \bb P^1$ is a deformation retract of $\bb C \bb P^2 \setminus \lbrace \ast \rbrace $ and the inclusion map $\bb C \bb P^1\hookrightarrow \bb C \bb P^2 \setminus \lbrace \ast \rbrace $  is a  cofibration 
with $\bb C \bb P^1$ a closed subset, it follows from \cite[Theorem 12]{S68}  that $\rho^{-1}(\bb C \bb P^1) \hookrightarrow \SU(3) \setminus \U(2)_{ij}$ is a closed cofibration as well. From this and \cite[Theorem~4]{S66}, we deduce that $\rho^{-1}(\bb C \bb P^1) $ is a deformation retract of $\SU(3) \setminus \U(2)_{ij}$. Note that $\rho^{-1}(\bb C \bb P^1)$ is a compact $6$-dimensional submanifold of $\SU(3) \setminus \U(2)_{ij}$. 
Furthermore, $\rho^{-1}(\bb C \bb P^1)$ is invariant under the $T^2$-action and hence  $\rho^{-1}(\bb C \bb P^1) \hookrightarrow \SU(3) \setminus \U(2)_{ij}$ is a $T^2$-equivariant map. This implies that the induced map between the Borel constructions $ET \times_T \rho^{-1}(\bb C \bb P^1)$  and $ET\times_T(\SU(3) \setminus \U(2)_{ij})$ is a homotopy equivalence (see for example \cite[Lemma~A.4.3., p.379]{AF24}). In particular,  $H^i_{T^2}(\rho^{-1}(\bb C \bb P^1))\cong H^i_{T^2}(\SU(3) \setminus \U(2)_{ij})$, for all $i$. To conclude  that $H^5((\SU(3)\setminus \U(2)_{ij})/T^2) = 0$ note that the $T^2$-action on  $\rho^{-1}(\bb C \bb P^1)$  is free.  Therefore,  $\rho^{-1}(\bb C \bb P^1)/T^2$ 
a $4$-dimensional closed manifold. This implies that   $H^i((\SU(3)\setminus \U(2)_{ij})/T^2)$ vanishes  for $i\geq 5$, as desired. Then from Proposition~\ref{prop:isom}, we have $H^i(\EO)= H^{i-4}_{T^2}(\U(2)_{ij})= H^{i-4}_\mathcal{O}(\bb S^2)$, for $i\geq 6$. By Theorem~\ref{thm:main:cohom}, the only remaining cohomology group we need to compute is the fourth cohomology group. We proceed with this computation as follows.

        Consider the following exact sequence from the proof of Proposition~\ref{prop:isom}:
         \begin{align*}
            \ldots \to H^4_T(\SU(3),\SU(3)\backslash \U(2)) \to H^4_\mathcal{O}(\EO) \to H^4_T(\SU(3) \backslash \U(2)) \to H^5_T(\SU(3),\SU(3)\backslash \U(2))\to \ldots
        \end{align*}
        Note that $H^{i}_T(\SU(3),\SU(3)\backslash\U(2)) \cong H^{i}_T(\mathcal{T}\U(2),\mathcal T \U(2)\backslash\U(2)) \cong H^{i-4}_T(\U(2))$, where the last isomorphism is the Thom isomorphism. It is clear that the odd degree cohomology groups of the pair $(\SU(3),\SU(3)\backslash\U(2))$ vanishes and that $H_T^4(\SU(3),\SU(3)\backslash\U(2)) \cong \bb Z$. Furthermore, since $(\SU(3) \backslash \U(2))/T^2$ is homotopy equivalent to a closed $4$-manifold, this exact sequence simplifies to:
         \begin{align*}
            \ldots \to \bb Z \to H^4_\mathcal{O}(\EO) \to \bb Z \to 0.
        \end{align*} 
        By Theorem~\ref{thm:main:cohom}, we have $H^4_\mathcal{O}(\EO) = \bb Z ^2 \oplus \bb Z_r$, for some $r \in \bb N$. The exact sequence  then implies  $r = 1$.
        By this we finish the proof of Case~\eqref{thm:cohomology:singsphere}. 
\end{proof}

If the singular set just consists of one point, one can easily deduce from the proof of Theorem~\ref{thm:main:sing}~\eqref{thm:main:sing:smoothsphere}, that the order of its local group is $k = 3,5$. Therefore we have the following Corollary from Remark~\ref{rem:singpoint} and Theorem~\ref{thm:cohomology:singsphere}:
\begin{corollary}
    If $\EO$ has a single singular point, then its isotropy group $\Z_k$ has order $3$ or $5$ and 
     \begin{align*}
                H^i_{\mathcal{O}}(\EO) = \begin{cases}
                0 & \text{for } i \text{ odd}\\
                    \bb Z & \text{for } i = 0,6\\
                    \bb Z^2 &\text{for } i = 2,4\\
                    \bb Z_k & \text{for } i\ge 8, \text{ even}
                \end{cases}
            \end{align*}
\end{corollary}

Theorem~\ref{thm:main:cohomgroups} is now a Corollary of Theorem~\ref{thm:cohomology}, Theorem~\ref{thm:main:sing}, Lemma~\ref{lem:actionssmooth} and the following Lemma.

\begin{lemma} \label{lem:cohom_smooth_sphere}
 The smooth orbifold spheres from Lemma \ref{lem:actionssmooth} have vanishing odd cohomology. Furthermore, $H^0_\mathcal{O}(S^2)=\bb Z$ and the remaining cohomology groups in even degrees are given by
    \begin{enumerate}
     \begin{multicols}{2}
        \item   \begin{align*}
            H_\mathcal{O}^n(S^2) = \begin{cases}
                \bb Z \oplus \bb Z_{k} & \text{for } n =2\\
                \bb Z_{k}^2 & \text{for } n= 2m \ge 4
            \end{cases}
        \end{align*}      
        in cases \eqref{lem:actionssmooth:1}, \eqref{lem:actionssmooth:2}, \eqref{lem:actionssmooth:nonpos:1} and~\eqref{lem:actionssmooth:nonpos:2} of Lemma~\ref{lem:actionssmooth}.
        \item  \begin{align*}
            H_\mathcal{O}^n(S^2) = \begin{cases}
                \bb Z \oplus \bb Z_{\gcd(2,k)} & \text{for } n =2\\
                \bb Z_{\gcd(2,k)} \oplus \bb Z_{k^2\over \gcd(2,k)} & \text{for } n= 2m \ge 4
            \end{cases}
        \end{align*}   
        in cases \eqref{lem:actionssmooth:5}, \eqref{lem:actionssmooth:6} of Lemma~\ref{lem:actionssmooth}.
        \end{multicols}   
    \end{enumerate}
    \end{lemma}

        \begin{proof}
            We use Theorem~\ref{thm:cohom:esch} to compute the cohomology rings of an orbifold $S^2$ given by a $T^2$-biquotient action on $\U(2)$. In fact, by just following the arguments of the proof of Theorem~\ref{thm:main:cohom}, we get
            \begin{align*}
                H_\mathcal{O}^\ast(\bb S^2) = \bb Z [t,s]/\langle \tilde \sigma_i(t\tilde p + s\tilde a) -\tilde \sigma_i(t\tilde q+s\tilde b)\vert \, i=1,2\rangle
            \end{align*}
            with
            \begin{align*}
                \tilde\sigma_1(T_1,T_2)&=T_1+T_2\\
                \tilde \sigma_2(T_1,T_2)&=T_1T_2
            \end{align*}
            and $\tilde p$, $\tilde q$, $\tilde a$, $\tilde b$ being the parameters of the $T^2$-action on $\U(2)$. In the cases of Lemma \ref{lem:actionssmooth}, we have $\U(2) = \U(2)_{11}$ and the parameters are given by the $(2, 3)$-components of the parameters of the $\T$-action on $\SU(3)$.  It is immediately clear that $H^{\mathrm{odd}}_\mathcal{O}(\bb S^2) = 0$. Furthermore, in all cases it is clear, that there is an $S^1 \subset T^2$, which acts freely. Therefore by \cite{tD87} (4.4) $\cup\, r \colon H^i_T(\bb S^2) \to H^{i+2}(\bb S^2)$ is an isomorphism for $i \ge 4$, where $r$ is the generator of $H^\ast(BS^1) \cong \bb Z[r]$. Hence we only need to compute the cohomology groups in degrees $2$ and $4$. The relation in degree $2$ is
            \begin{align*}
                \tilde\sigma_1(\tilde p - \tilde q)t + \tilde \sigma_1(\tilde a - \tilde b) s
            \end{align*}
            and the relations in degree $4$ are
             \begin{align*}
                &t(\tilde\sigma_1(\tilde p - \tilde q)t + \tilde \sigma_1(\tilde a - \tilde b)s)\\
                &s(\tilde\sigma_1(\tilde p - \tilde q)t + \tilde \sigma_1(\tilde a - \tilde b)s)\\
                &(\tilde\sigma_2(\tilde p) - \tilde \sigma_2(\tilde q))t^2 + (p_1a_2 + p_2 a_1 - q_1b_2 - q_2 b_1)st +( \tilde \sigma_2(\tilde a) - \tilde \sigma_2(\tilde b))s^2
            \end{align*}
            Therefore $H^2_\mathcal{O}(\bb S^2) \cong \bb Z \oplus \bb Z_{\gcd(\tilde\sigma_1(\tilde p - \tilde q),\tilde \sigma_1(\tilde a - \tilde b))}$ and $H^4_\mathcal{O}(\bb S^2) \cong \bb Z^3/\mathrm{Im}\,A$, with 
            \begin{align*}
                A = \begin{pmatrix}
                    \tilde\sigma_1(\tilde p - \tilde q) & \tilde\sigma_1(\tilde a - \tilde b)& 0\\
                    0&  \tilde\sigma_1(\tilde p - \tilde q) & \tilde\sigma_1(\tilde a - \tilde b)\\
                    \tilde\sigma_2(\tilde p) - \tilde \sigma_2(\tilde q) & \tilde p_1 \tilde a_2 + \tilde p_2 \tilde a_1 - \tilde q_1\tilde b_2 - \tilde  q_2 \tilde b_1 &\tilde \sigma_2(\tilde a) - \tilde \sigma_2(\tilde b)
                \end{pmatrix}
            \end{align*}
            The strategy is now to compute the Smith normal form of $A$ 
            \begin{align*}
                \Smith(A) = \begin{pmatrix}
                    \alpha_1 &&\\
                    &\alpha_2&\\
                    && \alpha_3
                \end{pmatrix}
            \end{align*}
            Then $\bb Z^3 /\mathrm{Im}\, A\cong \bb Z_{\alpha_1} \oplus \bb Z_{\alpha_2} \oplus \bb Z_{\alpha_3}$. $\Smith(A)$ can be obtained by applying basic operations on the columns and rows of $A$. Furthermore $\alpha_i = {d_i(A)\over d_{i-1}(A)}$, where $d_i(A)$ is the gcd of the $i \times i$ minors of $A$.
            \begin{enumerate}
                \item 
                \begin{align*}
                    A = \begin{pmatrix}
                        0 & 1-u & 0\\
                        0&0&1-u\\
                        -1 & 1- 2s + u & -u^2 +s (1-s+u)
                    \end{pmatrix} \text{ and }  \Smith(A) = \begin{pmatrix}
                    1&0&0\\
                       0&  1-u & 0\\
                       0& 0&1-u\\
                    \end{pmatrix}
                \end{align*}
                It follows immediately that $H^2_\mathcal{O}(\bb S^2) \cong \bb Z \oplus \bb Z_{1-u}$ and $H^2_\mathcal{O}(\bb S^2) \cong \bb Z_{1-u}^2$.
                \item 
                 \begin{align*}
                    A = \begin{pmatrix}
                        cr & -mr & 0\\
                        0&cr&-mr\\
                        e(cr -e) & r(cm(r+1)-1-2emr) - 2 e {em+1 \over c}& -m^2r(r+1) - {(em+1)^2\over c^2} + kr {em+1 \over c}
                    \end{pmatrix} 
                    \end{align*}
                    It is immediately clear, since $m$ and $c$ are relatively prime, that $H^2_\mathcal{O}(\bb S^2) \cong \bb Z \oplus \bb Z_r$. We multiply $A$ from the left with the matrix
                    \begin{align*}
                        \begin{pmatrix}
                            1&0&0\\
                            0&1&0\\
                            0&{em+1 \over c}& 1
                        \end{pmatrix} \cdot \begin{pmatrix}
                            1&0&0\\
                            0&1&0\\
                            0&-m(r+1)& 1
                        \end{pmatrix}
                    \end{align*}
                    to obtain
                    \begin{align*}
                        \tilde A = \begin{pmatrix}
                            cr & -mr & 0\\
                        0&cr&-mr\\
                        e(cr -e) & 2e{em+1 \over c}-emr & -{(1+em)^2 \over c^2}
                        \end{pmatrix} 
                    \end{align*}
                    Since the gcd of all entries needs to divide $r$, it also has to divide $e$ by the $(3,1)$-element of the matrix. But then also ${1+em \over c}$. Therefore $g = 1$, since $e$ and ${1+em \over c}$ are relatively prime. Furthermore $c^2r^2$ and $m^2r^2$ are $2$-minors of $\tilde A$. Therefore the gcd $g$ of all $2$-minors has to divide $r^2$. Since all other $2$-minors contain $r$ we have $g =rr^\prime$, where $r^\prime$ divides $r$. We can argue as for the $1$-minors, that $r^\prime =1$. Since $\det (\tilde A) = r^2$, we have  
                    \begin{align*}
                        \Smith(A) = \begin{pmatrix}
                            1&0&0\\
                            0&r&0\\
                            0&0&r
                        \end{pmatrix}.
                    \end{align*}
                    It follows that $H^1_{\mathcal{O}} \cong \bb Z_r^2$.
                    \item
                    \begin{align*}
                         A = \begin{pmatrix}
                        2 & -1-u & 0\\
                        0&2&-1-u\\
                        1 & -1-u & -u^2+s(1-s-1)
                    \end{pmatrix}
                    \end{align*}
                    It follows immediately that $H^2_\mathcal{O}(\bb S^2) \cong \bb Z \oplus \bb Z_{\gcd(2,u+1)}$. Since $\gcd(2,u+1) = \gcd(2,u-2s+1)$, the result follows in degree $2$. Furthermore it is clear, that $d_1(A) =1$ and $d_2(A) = \gcd(2,u+1)$, since $2, -u-1$ are $2$-minors and all other $2$-minors are divisible by $\gcd(2,u+1)$. Since $\det A = (u-2s-1)^2$, we obtain 
                    \begin{align*}
                        \Smith(A) =  \begin{pmatrix}
                            1&0&0\\
                            0&\gcd(2, u-2s +1)&0\\
                            0&0&{(u-2s+1)^2 \over \gcd(2, u-2s+1)}
                        \end{pmatrix}
                    \end{align*}
                    and the result follows.
                    \item 
                    \begin{align*}
                        A = \begin{pmatrix}
                        2e-cr  & -mr -2 {1-em \over c} & 0\\
                        0&2e-cr  & -mr -2 {1-em \over c}\\
                        e(cr -e) & -r + 2 e {1-em \over c}&  {(em+1)\over c} \left(3{1-em \over c} + mr\right)
                    \end{pmatrix} 
                    \end{align*}
                    Again we multiply $A$ with the matrix
                     \begin{align*}
                        \begin{pmatrix}
                            1&0&0\\
                            0&1&0\\
                            0&{em+1 \over c}& 1
                        \end{pmatrix} \cdot \begin{pmatrix}
                            1&0&0\\
                            0&1&0\\
                            0&-m(r+1)& 1
                        \end{pmatrix}
                    \end{align*}
                    to obtain 
                    \begin{align*}
                        \tilde A = \begin{pmatrix}
                        2e-cr  & -mr -2 {1-em \over c} & 0\\
                        0&2e-cr  & -mr -2 {1-em \over c}\\
                        e(cr -e) & -emr&  {(em-1)^2\over c^2}
                    \end{pmatrix} 
                    \end{align*} Since $c$ and $m$ are relatively prime, we find $s,t \in \bb Z$, such that $sc+tk = 1$. Then 
                    \begin{align*}
                        (2e-cr, -mr +2 {1-em \over c}) \cdot \begin{pmatrix}
                            s & -k\\
                            t & c
                        \end{pmatrix} = (2es -r + 2t  {1-em \over c}, 2)
                    \end{align*}
                    Therefore $H^2_\mathcal{O}(\bb S^2) \cong \bb Z \oplus \bb Z_{\gcd(2,r)}$. 
                    Since $\gcd\left(e, {(em-1)\over c}\right) = \gcd\left(m, {(em+1)\over c}\right) = 1$, we have $d_1(\tilde A) =1$. Since $\gcd\left(2e-cr, -mr +2 {1-em \over c}\right) = \gcd(2,r)$ and $(2e-cr)^2$ and $(-mr -2 {1-em \over c})^2$ are $2$-minors, it is easy to see that $d_2(\tilde A) = \gcd(2,r)$. Since $\det(\tilde A) = -r^2$, we have
                    \begin{align*}
                        \Smith(A) = \begin{pmatrix}
                            1&0&0\\
                            0&\gcd(2,r)&0\\
                            0&0& {r^2 \over \gcd(2,r)}
                        \end{pmatrix}
                    \end{align*}
                   
                    \item  
                \begin{align*}
                    A = \begin{pmatrix}
                        -2 & 2a_2 & 0\\
                        0&-2&2a_2\\
                        0 & -2a_2& 1-2a_2^2
                    \end{pmatrix} \text{ and }  \Smith(A) = \begin{pmatrix}
                    1&0&0\\
                       0&  2 & 0\\
                       0& 0&2\\
                    \end{pmatrix}
                \end{align*}
                It follows immediately that $H^2_\mathcal{O}(\bb S^2) \cong \bb Z \oplus \bb Z_{2}$ and $H^2_\mathcal{O}(\bb S^2) \cong \bb Z_{2}^2$.
                    \item Since $\vert a_1 \vert = 2$, we have
                \begin{align*}
                    A = \begin{pmatrix}
                        0 & -a_1 & 0\\
                        0&0&-a_1\\
                        -1 & (a_1-2a_2)x& -3+a_1a_2-a_2^2
                    \end{pmatrix} \text{ and }  \Smith(A) = \begin{pmatrix}
                    1&0&0\\
                       0&  2 & 0\\
                       0& 0&2\\
                    \end{pmatrix}
                \end{align*}
                It follows immediately that $H^2_\mathcal{O}(\bb S^2) \cong \bb Z \oplus \bb Z_{2}$ and $H^2_\mathcal{O}(\bb S^2) \cong \bb Z_{2}^2$.
            \end{enumerate}
        \end{proof}

\begin{remark}
    Let $\bb S^2 \cong \U(2)\sslash T$ be an orbifold $2$-sphere with orbifold groups $\bb Z_{k_1}$, $\bb Z_{k_2}$ at the singular orbits and $\bb Z_k$ at the pricipal orbits. It easily follows from a Mayer-Vietoris-Sequence, that the stable orbifold cohomology group of $\bb S^2$  lies in a short exact sequence
    \begin{align*}
        0 \to \Z_k \to H^4_\OO(\bb S^2) \to \Z_{k_1k_2 \over k} \to 0
    \end{align*}
    Therefore the order of $H^4_\OO(\bb S^2)$ is $k_1k_2$
\end{remark}
\begin{example} \label{ex:nonegative}
    Consider 
    \begin{align*}
        p = (1,-1,-1), \quad q =(-1,0,0), \quad a = (l-1+2u+k,1,-l), \quad b=(0,u,u+k)
    \end{align*}
    For $k, u \in \Z$, $k \neq \pm 2$,  $l = \pm 1$. The singular Set of $\EO$ is an orbifold $2$-sphere with orbifold groups $\bb Z_{1+k+l}$, $\Z_{1-k+l}$ and $\Z_{\gcd(k,l+1)}$. From the proof of Theorem~\ref{thm:main:sing} it follows, that $\EO$ has no Eschenburg metric with positive sectional curvature, if and only if $l = 1$ and $k \neq 0$. Using the method of Lemma \ref{lem:cohom_smooth_sphere} is easy to compute that its cohomology groups are given by
    \begin{align*}
          H^i_{\mathcal{O}}(\EO) = \begin{cases}
                0 & \text{for } i \text{ odd}\\
                    \bb Z & \text{for } i = 0\\
                    \bb Z^2 &\text{for } i = 2,4\\
                    \Z \oplus \Z_{\gcd (2,k)}  & \text{for } i = 6\\
                    \Z_{\gcd (2,k)} \oplus \Z_{k^2-(1+l)^2\over \gcd(2,k)} & \text{for } i = 8
                \end{cases}
    \end{align*}
    If $l = 1$, then the order of $H^8_{\mathcal{O}}(\EO)$ is $k^2-4$, which cannot be a square, unless $k=0$, and therefore cannot be realised by a positively curved Eschenburg Orbifold.
    Note that $\EO$ has exactly two singular points, if $l =1$ and $k$ is odd.
\end{example}

\bibliographystyle{alpha}
\bibliography{References}

\newcommand{\etalchar}[1]{$^{#1}$}
\begin{thebibliography}{GGK{\etalchar{+}}24}

\bibitem[AF24]{AF24}
David Anderson and William Fulton.
\newblock {\em Equivariant cohomology in algebraic geometry}, volume 210 of
  {\em Cambridge Studies in Advanced Mathematics}.
\newblock Cambridge University Press, Cambridge, 2024.

\bibitem[ALR07]{ALR07}
Alejandro Adem, Johann Leida, and Yongbin Ruan.
\newblock {\em Orbifolds and Stringy Topology}.
\newblock Cambridge Tracts in Mathematics. Cambridge University Press, 2007.

\bibitem[AMP97]{ASTEY199741}
L~Astey, E~Micha, and G~Pastor.
\newblock Homeomorphism and diffeomorphism types of eschenburg spaces.
\newblock {\em Differential Geometry and its Applications}, 7(1):41--50, 1997.

\bibitem[Baz96]{Bazaikin1996}
Ya~V. Bazaikin.
\newblock On a certain family of closed 13-dimensional riemannian manifolds of
  positive curvature.
\newblock {\em Siberian Mathematical Journal}, 37(6):1068–1085, November
  1996.

\bibitem[Bor53]{B53}
Armand Borel.
\newblock Sur la cohomologie des espaces fibres principaux et des espaces
  homogenes de groupes de lie compacts.
\newblock {\em The Annals of Mathematics}, 57(1):115, January 1953.

\bibitem[CEZ07]{Chinburg2007}
Ted Chinburg, Christine Escher, and Wolfgang Ziller.
\newblock Topological properties of eschenburg spaces and 3-sasakian manifolds.
\newblock {\em Mathematische Annalen}, 339(1):3–20, April 2007.

\bibitem[CJ21]{Car22}
Francisco~C. Caramello~Jr.
\newblock Introduction to orbifolds.
\newblock {\em \href{https://arxiv.org/pdf/1909.08699}{ arXiv:1909.08699v6
  [math.DG] }}, 2021.

\bibitem[Dea11]{Dearricott2011}
Owen Dearricott.
\newblock A 7-manifold with positive curvature.
\newblock {\em Duke Mathematical Journal}, 158(2), June 2011.

\bibitem[DJ23]{DeVJoh23}
Jason DeVito and Peyton Johnson.
\newblock Almost positively curved eschenburg spaces, 2023.

\bibitem[Esc82]{Esch82}
J.~H. Eschenburg.
\newblock New examples of manifolds with strictly positive curvature.
\newblock {\em Inventiones Mathematicae}, 66(3):469–480, October 1982.

\bibitem[Esc92]{Esch92}
J.-H. Eschenburg.
\newblock Cohomology of biquotients.
\newblock {\em Manuscripta Math.}, 75(2):151--166, 1992.

\bibitem[FZ07]{FZ07}
Luis~A. Florit and Wolfgang Ziller.
\newblock Orbifold fibrations of {E}schenburg spaces.
\newblock {\em Geom. Dedicata}, 127:159--175, 2007.

\bibitem[FZ09]{Florit2009}
Luis~A. Florit and Wolfgang Ziller.
\newblock On the topology of positively curved bazaikin spaces.
\newblock {\em Journal of the European Mathematical Society}, 11(1):189–205,
  February 2009.

\bibitem[GGK{\etalchar{+}}24]{Getal24}
Katie Gittins, Carolyn Gordon, Magda Khalile, Ingrid Membrillo~Solis, Mary
  Sandoval, and Elizabeth Stanhope.
\newblock Do the {H}odge spectra distinguish orbifolds from manifolds? {P}art
  1.
\newblock {\em Michigan Math. J.}, 74(3):571--598, 2024.

\bibitem[GM74]{GM74}
Detlef Gromoll and Wolfgang Meyer.
\newblock An exotic sphere with nonnegative sectional curvature.
\newblock {\em Ann. of Math. (2)}, 100:401--406, 1974.

\bibitem[GVZ11]{Grove2011}
Karsten Grove, Luigi Verdiani, and Wolfgang Ziller.
\newblock An exotic ${T_{1}\mathbb{S}^4}$ with positive curvature.
\newblock {\em Geometric and Functional Analysis}, 21(3):499–524, April 2011.

\bibitem[GW22]{GW22}
Oliver Goertsches and Michael Wiemeler.
\newblock Non-negatively curved {GKM} orbifolds.
\newblock {\em Math. Z.}, 300(2):2007--2036, 2022.

\bibitem[GWZ08]{gwz}
K.~Grove, B.~Wilking, and W.~Ziller.
\newblock {Positively curved cohomogeneity one manifolds and $3$-Sasakian
  geometry}.
\newblock {\em J. Differential Geom.}, 78:33--111, 2008.

\bibitem[Ker11]{Kerin2011}
Martin Kerin.
\newblock On the curvature of biquotients.
\newblock {\em Mathematische Annalen}, 352(1):155–178, January 2011.

\bibitem[KL14]{KL14}
Bruce Kleiner and John Lott.
\newblock Geometrization of three-dimensional orbifolds via {Ricci} flow.
\newblock In {\em Local collapsing, orbifolds, and geometrization}, number 365
  in Ast\'erisque, pages 101--177. Soci\'et\'e math\'ematique de France, 2014.

\bibitem[KS91]{Kreck1991}
Matthias Kreck and Stephan Stolz.
\newblock Some nondiffeomorphic homeomorphic homogeneous 7-manifolds with
  positive sectional curvature.
\newblock {\em Journal of Differential Geometry}, 33(2), January 1991.

\bibitem[KWW21]{Kennard2021-ge}
Lee Kennard, Michael Wiemeler, and Burkhard Wilking.
\newblock Splitting of torus representations and applications in the grove
  symmetry program.
\newblock {\em \href{https://arxiv.org/abs/2106.14723}{ arXiv:2106.14723v1
  [math.DG] }}, June 2021.

\bibitem[Lan18]{Lange2018}
Christian Lange.
\newblock On metrics on 2-orbifolds all of whose geodesics are closed.
\newblock {\em Journal f\"{u}r die reine und angewandte Mathematik (Crelles
  Journal)}, 2020(758):67–94, January 2018.

\bibitem[LMCG97]{lin1997collision}
Ming~C Lin, Dinesh Manocha, Jon Cohen, and Stefan Gottschalk.
\newblock Collision detection: Algorithms and applications.
\newblock {\em Algorithms for robotic motion and manipulation}, pages 129--142,
  1997.

\bibitem[LPK{\etalchar{+}}21]{Lazaridis2021}
Lazaros Lazaridis, Maria Papatsimouli, Konstantinos-Filippos Kollias,
  Panagiotis Sarigiannidis, and George~F. Fragulis.
\newblock {\em Hitboxes: A Survey About Collision Detection in Video Games},
  page 314–326.
\newblock Springer International Publishing, 2021.

\bibitem[Nie22]{Nienhaus22}
Jan Nienhaus.
\newblock An improved four-periodicity theorem and a conjecture of hopf with
  symmetry.
\newblock {\em \href{https://arxiv.org/abs/2211.13151}{ arXiv:2211.13151v1
  [math.DG] }}, 2022.

\bibitem[Sch14]{Sch14}
Rolf Schneider.
\newblock {\em Convex bodies: the {B}runn-{M}inkowski theory}, volume 151 of
  {\em Encyclopedia of Mathematics and its Applications}.
\newblock Cambridge University Press, Cambridge, expanded edition, 2014.

\bibitem[Sm66]{S66}
Arne Str\o~m.
\newblock Note on cofibrations.
\newblock {\em Math. Scand.}, 19:11--14, 1966.

\bibitem[Sm68]{S68}
Arne Str\o~m.
\newblock Note on cofibrations. {II}.
\newblock {\em Math. Scand.}, 22:130--142 (1969), 1968.

\bibitem[Sta05]{Stanhope2005}
Elizabeth Stanhope.
\newblock Spectral bounds on orbifold isotropy.
\newblock {\em Annals of Global Analysis and Geometry}, 27(4):355–375, June
  2005.

\bibitem[tD87]{tD87}
Tammo tom Dieck.
\newblock {\em Transformation Groups}.
\newblock De Gruyter, Berlin, New York, 1987.

\bibitem[Wil02]{Wilking2002}
Burkhard Wilking.
\newblock Manifolds with positive sectional curvature almost everywhere.
\newblock {\em Inventiones Mathematicae}, 148(1):117–141, April 2002.

\bibitem[Wil03]{Wilking2003}
Burkhard Wilking.
\newblock Torus actions on manifolds of positive sectional curvature.
\newblock {\em Acta Mathematica}, 191(2):259–297, 2003.

\bibitem[Wil06]{Wilking2006}
Burkhard Wilking.
\newblock Positively curved manifolds with symmetry.
\newblock {\em Annals of Mathematics}, 163(2):607–668, March 2006.

\bibitem[Wul23]{Wulle2023}
Dennis Wulle.
\newblock Cohomogeneity one manifolds with quasipositive curvature.
\newblock {\em Mathematische Annalen}, 390(1):303–350, November 2023.

\bibitem[WZ18]{WilkingZiller}
Burkhard Wilking and Wolfgang Ziller.
\newblock Revisiting homogeneous spaces with positive curvature.
\newblock {\em J. Reine Angew. Math.}, 738:313--328, 2018.

\bibitem[Yer14a]{Yer14}
Dmytro Yeroshkin.
\newblock On the geometry and topology of $4$-orbifolds.
\newblock {\em \href{https://arxiv.org/abs/1411.1700}{ arXiv:1411.1700v1
  [math.DG] }}, 2014.

\bibitem[Yer14b]{Y14}
Dmytro Yeroshkin.
\newblock {\em Riemannian orbifolds with non-negative curvature}.
\newblock ProQuest LLC, Ann Arbor, MI, 2014.
\newblock Thesis (Ph.D.)--University of Pennsylvania.

\bibitem[Yer15]{Y15}
Dmytro Yeroshkin.
\newblock Orbifold biquotients of {$SU(3)$}.
\newblock {\em Differential Geom. Appl.}, 42:54--76, 2015.

\bibitem[Zil14]{Zil14}
Wolfgang Ziller.
\newblock {\em Riemannian Manifolds with Positive Sectional Curvature}, page
  1–19.
\newblock Springer International Publishing, 2014.

\end{thebibliography}

\end{document}